\documentclass[11pt,reqno]{amsart}
\usepackage[T1]{fontenc}
\usepackage[utf8]{inputenc}
\usepackage{amsfonts,amssymb,amsmath,amsthm,cancel}
\usepackage{bbding}
\usepackage{slashed}
\usepackage{todonotes}
\usepackage{amsrefs}
\usepackage{a4wide}
\usepackage{enumerate} 
\usepackage{enumitem}
\usepackage{mathrsfs}
\usepackage{bbm}
\usepackage{nicefrac}
\usepackage{upgreek}
\usepackage{tikz} 
\usetikzlibrary{arrows}
\usepackage{caption}
\usepackage{hyperref}
\hypersetup{
    colorlinks,
    linkcolor={green!50!black},
    citecolor={blue},
    urlcolor={cyan!50!blue}
}
\newcommand{\sM}{M}
\newcommand{\sC}{\mathscr{C}}
\newcommand{\Hom}[1]{\mathrm{Hom} ({#1})}
\DeclareMathOperator{\End}{End}
\DeclareMathOperator{\Tr}{Tr}
\DeclareMathOperator{\tr}{tr}
\DeclareMathOperator{\rk}{rk}
\DeclareMathOperator{\res}{res}
\DeclareMathOperator{\codim}{codim}

\newcommand{\N}{\mathbb{N}}
\newcommand{\Z}{\mathbb{Z}}
\newcommand{\R}{\mathbb{R}}
\newcommand{\C}{\mathbb{C}}
\newcommand{\bbL}{\mathbb{L}}

\DeclareMathOperator{\supp}{supp}  
\DeclareMathOperator{\singsupp}{singsupp}

\newcommand{\grad}{\mathrm{grad}}

\newcommand{\one}{\mathbbm{1}}
\DeclareMathOperator{\WF}{WF}

\newcommand{\km}{\mathfrak{m}}
\DeclareMathOperator{\Hol}{Hol} 
\DeclareMathOperator{\cst}{cst} 
\newcommand{\cT}{\mathcal{T}}

\newcommand{\E}{\mathbb{E}}
\newcommand{\sH}{\mathscr{H}}
\newcommand{\halfDen}{ \varOmega^{\nicefrac{1}{2}} }
\DeclareMathOperator{\ret}{ret}
\DeclareMathOperator{\adv}{adv}

\newcommand{\spacetime}{(\sM, \fg)}
\newcommand{\SST}{(\sM, \fg, \varXi)}

\newcommand{\secECauchy}{C^{\infty} (\varSigma; \sEsM_{\varSigma})}

\newcommand{\secsME}{C^{\infty} (\sM; \sEsM)}
\newcommand{\comSecsME}{C_{\mathrm{c}}^{\infty} (\sM; \sEsM)}
\newcommand{\dualComSecsME}{\mathcal{D}^{\prime} (\sM; \sEsM)}

\newcommand{\comSecE}{C_{\mathrm{c}}^{\infty} (M; E)}

\newcommand{\advGreenOp}{G^{\mathrm{adv}}} 
\newcommand{\retGreenOp}{G^{\mathrm{ret}}}
\newcommand{\coTanM}{\mathrm{T}^{*} M}

\newcommand{\coTansM}{\mathrm{T}^{*} \! \sM}
\newcommand{\coTanCauchy}{\mathrm{T}^{*} \! \varSigma}

\newcommand{\dotCoTan}{\dot{\mathrm{T}}^{*}}
\newcommand{\coTan}{\mathrm{T}^{*}}
\newcommand{\tangent}{\mathrm{T}}
\newcommand{\dotCoTanM}{\dot{\mathrm{T}}^{*} \! M}
\newcommand{\dotCoTansM}{\dot{\mathrm{T}}^{*} \! \sM}
\newcommand{\dotCoTanN}{\dot{\mathrm{T}}^{*} N}

\newcommand{\dotCoTanMN}{\dot{\mathrm{T}}^{*} \! (M \times N)}
\newcommand{\dotCoTanCauchy}{\dot{\mathrm{T}}^{*} \! \varSigma}

\newcommand{\coLightBun}{\dot{\mathrm{T}}_{0}^{*} \sM}
\newcommand{\ms}{\scriptscriptstyle}
\newcommand{\pr}{\mathrm{pr}}
\newcommand{\Rn}{\mathbb{R}^{n}} 
\newcommand{\Rd}{\mathbb{R}^{d}}

\newcommand{\xxiNot}{(x_{0}, \xi^{0})}

\newcommand{\PsiDO}[2]{\Psi \mathrm{DO}^{{#1}} ({#2})}
 
\DeclareMathOperator{\rc}{c}
\DeclareMathOperator{\ri}{i}
\newcommand{\B}{\mathbb{B}}
\newcommand{\re}{\mathrm{e}}
\newcommand{\rd}{\mathrm{d}} 

\newcommand{\bbl}{\mathbbm{l}}

\newcommand{\cF}{\mathcal{F}}

\newcommand{\cH}{\mathcal{H}}
\newcommand{\Maslov}{\mathbb{M}}

\newcommand{\xytheta}{(x, y; \theta)}

\newcommand{\xxi}{(x, \xi)}

\newcommand{\yeta}{(y, \eta)}
\newcommand{\xxiyeta}{(x, \xi; y, \eta)}

\newcommand{\parDeri}[2]{\frac{\partial {#2}}{\partial {#1}}}
\newcommand{\fA}{\mathsf{A}} 
\newcommand{\fa}{\mathsf{a}}

\newcommand{\sE}{E}
\newcommand{\sEsM}{\sE}

\newcommand{\fc}{\mathsf{c}}
\newcommand{\cD}{\mathcal{D}}

\newcommand{\fF}{\mathsf{F}}
\newcommand{\kF}{\mathfrak{F}}
\newcommand{\tE}{\mathtt{E}}
\newcommand{\fh}{h}
\newcommand{\fg}{g}

\newcommand{\fG}{\mathsf{G}}

\newcommand{\fK}{\mathsf{K}}
\newcommand{\fN}{\mathsf{N}}
\newcommand{\fR}{\mathsf{R}}
\newcommand{\cR}{\mathcal{R}} 
\newcommand{\cC}{\mathcal{C}}

\newcommand{\cN}{\mathcal{N}}

\newcommand{\cP}{\mathcal{P}}
\newcommand{\rP}{\mathrm{P}} 
\newcommand{\sR}{\mathscr{R}} 
 
\newcommand{\T}{\mathbb{T}}
\newcommand{\sgn}{\mathrm{sgn}}
\DeclareMathOperator{\Char}{Char}

\newcommand{\cV}{\mathcal{V}}

\newcommand{\sU}{\mathscr{U}}
\newcommand{\fU}{\mathsf{U}}

\newcommand{\M}{\mathbb{M}}
\newcommand{\m}{\mathbbm{m}}
\newcommand{\bbt}{\mathbbm{t}}

\newcommand{\symb}[1]{\sigma_{\! \ms{#1}}}
\newcommand{\subSymb}[1]{\sigma^{\mathrm{sub}}_{\ms{#1}}}

\newcommand{\dVolh}{\mathrm{d} \mathsf{v}_{\fh}} 
\newcommand{\dVolg}{\mathrm{d} \mathsf{v}_{\fg}}
\newcommand{\dVol}{\mathrm{d} \mathsf{v}} 
\newcommand{\vol}{\mathsf{vol}} 
\newcommand{\nM}{n_{\ms{M}}} 
\newcommand{\nN}{n_{\ms{N}}}

\DeclareMathOperator{\spec}{Spec}

\newcommand{\connectionE}{ \nabla^{\ms \sE} }

\newcommand{\connectionEndPiE}{ \nabla^{\ms \pi^{*} \Hom{\sE, \sE}} }

\newcommand{\scalarProdTwo}[2]{\left\langle {#1} \left| {#2} \right. \right\rangle}

\theoremstyle{plain}
\newtheorem{theorem}{Theorem}[section]
\newtheorem{proposition}[theorem]{Proposition}
\newtheorem{lemma}[theorem]{Lemma}
\newtheorem{corollary}[theorem]{Corollary}
\theoremstyle{definition}
\newtheorem{definition}[theorem]{Definition}

\newtheorem{remark}[theorem]{Remark}

\newtheorem{assumption}[theorem]{Assumption}
\allowdisplaybreaks
\title[Gutzwiller trace formula]{A Gutzwiller Trace formula for Dirac Operators on a Stationary Spacetime}
\author[O.~Islam]{Onirban Islam}
\address{School of Mathematics, University of Leeds, Leeds LS2 9JT, UK}
\address{Institut f\"{u}r Mathematik, Universit\"{a}t Potsdam, 14476 Potsdam, Germany}
\email{onirban.islam@math.uni-potsdam.de}
\date{10 August 2022}
\begin{document}
\begin{abstract}
	A Duistermaat-Guillemin-Gutzwiller trace formula for Dirac-type operators on a globally hyperbolic spatially compact stationary spacetime is achieved by generalising the recent construction by 
	A. Strohmaier and S. Zelditch [Adv. Math. \textbf{376}, 107434 (2021)] 
	to a vector bundle setting. 
	We have analysed the spectrum of the Lie derivative with respect to a global timelike Killing vector field on the solution space of the Dirac equation and found that it consists of discrete real eigenvalues. 
	The distributional trace of the time evolution operator has singularities at the periods of induced Killing flow on the space of lightlike geodesics. 
	This gives rise to the Weyl law asymptotic at the vanishing period.   
\end{abstract}
\maketitle
%
%
%
%
%
%
%
%
%
%
\section{Introduction}
Let $\Delta$ be the Laplacian (in Geometers' convention) on a $d-1$-dimensional closed Riemannian manifold $(\varSigma, \fh)$.  
It is well-known that its spectrum $\spec \Delta$ comprises discrete eigenvalues $\lambda_{n}^{2}$ and the counting function $\fN_{\Delta} (\lambda) := \# \{ n \in \N | \lambda_{n} \leq \lambda \}$ satisfies the celebrated 
Weyl law~\cite{Hoermander_ActaMath_1968} 
\begin{equation} \label{eq: Weyl_law_Laplacian}
	\fN_{\Delta} (\lambda) = \left( \frac{\lambda}{2 \pi} \right)^{d-1} \vol (\B^{d-1}) \, \vol (\varSigma), \quad \textrm{as} \quad \lambda \to \infty, 
\end{equation}
where $\vol (\B^{d-1})$ is the volume of the unit ball $\B^{d-1}$ in $\R^{d-1}$.  
The counting measure $\rd \fN_{\Delta} / \rd \lambda$ is related to $\Tr U_{t} (\sqrt{\Delta})$ via the distributional Fourier transform, where 
\begin{equation} \label{eq: def_U_t_Laplacian}
	U_{t} (\sqrt{\Delta}) := \re^{- \ri t \sqrt{\Delta}}, \quad \forall t \in \R
\end{equation}
makes sense as a distribution on $\R$. 
Thus, the information contained in $U_{t} (\sqrt{\Delta})$ can be proficiently transferred to the spectral function utilising Tauberian arguments to derive the Weyl law.  
Additionally, $U_{t} (\sqrt{\Delta})$ is the solution map to the Cauchy problem of the half-wave operator $- \ri \partial_{t} + \sqrt{\Delta}$ and it satisfies the group property. 
Moreover, the singularity analysis of $\Tr U_{t} (\sqrt{\Delta})$ at $t=0$ gives rise to the half-wave trace invariants and its singular support is contained in the set of periods of periodic geodesics $\gamma$ on $\varSigma$. 
Furthermore, $\Tr U_{t} (\sqrt{\Delta})$ admits a singularity expansion around the non-zero periods $t = T \neq 0$ (under some technical assumptions)~\cite{Chazarain_InventMath_1974, Duistermaat_Inventmath_1975} 
and its leading-order term has been expressed by means of the fundamental periods and eigenvalues of the Poincar\'{e} maps of $\gamma$'s by  
Duistermaat and Guillemin~\cite[Thm. 4.5]{Duistermaat_Inventmath_1975}.  
It is therefore evident that the spectral invariant $\Tr U_{t} (\sqrt{\Delta})$ serves as an efficient generating function to study spectral geometry. 

Sandoval has generalised the Duistermaat-Guillemin trace formula for Dirac-type operators $\hat{D}$ on a hermitian vector bundle $\sE \to \varSigma$ over a closed Riemannian manifold $(\varSigma, \fh)$ by studying 
\begin{equation} \label{eq: def_U_t_Riem_Dirac}
	U_{t} (\hat{D}) := \re^{- \ri t \hat{D}}, \quad \forall t \in \R.  
\end{equation}
Arguably, $\hat{D}$ is the most fundamental first-order differential operator in geometric analysis and it differs from a Laplace-Beltrami operator $\Delta$ on $\sE$ in that its spectrum $\spec \hat{D}$ is unbounded from both below and above whereas the $\spec \Delta$ is bounded from below (semibounded). 
Sandoval's investigation results in a straightforward generalisation of 
the Weyl law~\cite[$(5)$]{Sandoval_CPDE_1999}  
and Dirac-wave trace invariants~\cite[Thm. 2.2]{Sandoval_CPDE_1999}. 
But, genuine bundle features show up in terms of the holonomy group in the singularity analysis of $\Tr U_{T \neq 0} (\hat{D})$~\cite[Thm. 2.8]{Sandoval_CPDE_1999}. 
However, there is an ad-hoc choice regarding the connection used to deduce the holonomy group as elucidated below. 

As a matter of fact, the Weitzenb\"{o}ck connection $\hat{\nabla}$ (see Remark~\ref{rem: Clifford_connection}) determined by $\hat{D}$ does not automatically induce a Clifford connection albeit such a connection always exists and the corresponding operator is usually called a \textit{compatible Dirac-type operators} $\tilde{D}$. 
They differ only by a smooth term $W := \hat{D} - \tilde{D} \in C^{\infty} (\varSigma; \End \sE)$. 
In order to compute the principal symbol $\symb{U_{t} (\hat{D})}$ of $U_{t} (\hat{D})$, Sandoval has used the parallel transporter $\tilde{\cT}$ with respect to the Weitzenb\"{o}ck connection $\tilde{\nabla}$ of $\tilde{D}$ instead of $\hat{\cT}$ corresponding to $\hat{\nabla}$~\cite[Prop. 5.3]{Sandoval_CPDE_1999}.  
Moreover, she employed the trivialisation characterised by the vanishing subprincipal symbol $\subSymb{\tilde{D}^{2}}$ of $\tilde{D}^{2}$~\cite[Cor. 5.11]{Sandoval_CPDE_1999}. 
As a consequence, $\tilde{\cT}$ appearing in the expression of $\symb{U_{t} (\hat{D})}$ is somewhat expedient and the splitting of $\hat{\cT}$ in terms of $\tilde{\cT}$ and the 
average value\footnote{See~\cite[$(16)$]{Sandoval_CPDE_1999}  
	for the precise expression. 
} 
of $W$ along the geodesic $[0, t] \ni s \mapsto \gamma (s) \in \mathbb{S}^{*} \varSigma$ is rather ad-hoc.

The purpose of this article is two-fold. 
Firstly, we have relaxed all the aforementioned ad-hoc considerations to compute $\symb{U_{t} (\hat{D})}$.   
Next, one notices that $U_{t} (\hat{D})$ solves the initial value problem 
\begin{equation} \label{eq: Dirac_wave}
	(- \ri \partial_{t} + \hat{D}) u = 0, \quad u (t_{0}) = k  
\end{equation}
on the $d$-dimensional ultrastatic spacetime $(\R \times \varSigma, \rd t^{2} - \fh)$ (see Section~\ref{sec: GHSST}) admitting the timelike Killing vector field $Z := \partial_{t}$ so that the external parameter $t$ in $U_{t} (\hat{D})$ can be considered as the canonical global time-coordinate on $\R$. 
But, $- \ri \partial_{t} + \hat{D}$ does not respect general relativistic covariance and so it seems more natural to instead work with the Lorentzian Dirac-type operator $D$. 
Thus the article is devoted to investigate a Lorentzian generalisation of Sandoval's trace formula.

This Riemannian to Lorentzian transition is conceptually non-trivial because the solution map to the Cauchy problem of the Dirac equation 
\begin{equation} \label{eq: Dirac_eq}
	D u = 0, \quad u |_{\ms \varSigma} = k  
\end{equation} 
on a globally hyperbolic Lorentzian manifold is no longer of the form~\eqref{eq: def_U_t_Riem_Dirac} as the Dirac-Hamiltonian (the Lorentzian counterpart of $\hat{D}$) becomes time dependent. 
In order to contemplate what should be looked for in a non-ultrastatic setting, we revisit the ultrastatic scenario from a relativistic viewpoint. 
One reckons that the solutions of~\eqref{eq: Dirac_wave} are of the form 
\begin{equation}
	\uppsi_{n}^{\pm} := \re^{- \ri t \lambda_{n}^{\pm}} \upphi_{n}^{\pm}
\end{equation}
where $\upphi_{n}^{\pm}$ are the eigensections of $\hat{D}$ corresponding to the eigenvalues $\lambda_{n}^{\pm}$. 
But 
\begin{equation} \label{eq: eigenvalue_KVF}
	Z \uppsi_{n}^{\pm} = - \ri \lambda_{n}^{\pm} \uppsi_{n}^{\pm}  
\end{equation}
and consequently, 
\begin{equation} \label{eq: trace_time_evolution_op_tr_Killing_flow}
	\Tr_{L^{2} (\varSigma; \sE)} U_{t} (\hat{D}) 
	:= \sum_{\lambda_{n}^{\pm} \in \spec \hat{D}} \re^{- \ri t \lambda_{n}^{\pm}} 
	= \Tr_{\ker (- \ri \partial_{t} + \hat{D})} \re^{tZ}. 
\end{equation}
In other words, the \textit{trace of the Dirac-wave group $U_{t} (\hat{D})$ on a closed Riemannian manifold (Cauchy hypersurface) is equivalent to the trace of the flow $\re^{tZ}$ induced by the timelike Killing vector field $Z$ of an ultrastatic spacetime acting on the kernel $\ker (- \ri \partial_{t} + \hat{D})$ of the half-wave operator $- \ri \partial_{t} + \hat{D}$}. 
This is essentially a straightforward generalisation of the observation originally due to 
Strohmaier and Zelditch~\cite[Sec. 10.1]{Strohmaier_AdvMath_2021} 
in the context of a scalar wave operator.    

It is, thus, evident that a timelike Killing flow $\varXi$ plays a pivotal role and the most general class of spacetimes admitting such a flow is known as the stationary spacetime $\SST$.  Physically speaking, these spacetimes admit a canonical flow of time, but unlike the ultrastatic case, there is no preferred time coordinate. 
They are interesting because a couple of exact solutions to Einstein equation: the Schwarzschild and the Kerr black holes, belong to this class. 
We demand the global hyperbolicity condition on $\spacetime$ to ensure a well-posed Cauchy problem for $D$.  
Therefore, the object of our study is the distributional trace $\Tr U_{t}$ of the \textit{time evolution operator} $U_{t}$ on a compact (to guarantee discrete eigenvalues) Cauchy hypersurface $\varSigma$ without boundary. 
In addition to the conceptual issues, our computational techniques are different from those used by Sandoval, primarily due to the fact that the Dirac Hamiltonian $H_{D}$ (see~\eqref{eq: def_Dirac_Hamiltonian_SSST}) on a stationary spacetime is not of Dirac-type in contrast to $\hat{D}$. 
Furthermore, we have computed $\symb{U_{t}}$ using the Weitzenb\"{o}ck connection induced by $D$ in order to avoid the aforementioned expedient choice in Sandoval's work, as explained in the comment after the statement of Theorem~\ref{thm: trace_formula_t_T}. 

As a quintessential relativistic operator, it furthermore allows studying the semiclassical limit of quantum field theoretic observables interacting with classical 
gravity\footnote{We 
	do not require Einstein's field equations to be satisfied.
}. 
Whilst spin- or spin$^{\mathrm{c}}$-Dirac operators are desirable from a physics point of view; we prefer to work on more general Dirac-type operators because they do not enforce any topological restrictions on  
$M$\footnote{The 
	global existence of a Lorentz metric depends on the topology of $M$. 
	In particular, such a metric exists in all non-compact manifolds and compact manifolds with vanishing Euler characteristics.}
unlike the 
spin\footnote{A 
	spin-structure always exists locally, but its global existence depends on some higher orientability property of the base manifold. 
	For instance, spin$^{\mathrm{c}}$ (resp. spin) structure exists if and only if the third integral (resp. the second) Stiefel–Whitney class of $M$ vanish.}-Dirac 
operators. 
Furthermore, some higher-spin operators are of Dirac-type. 

We would like to close the introduction by mentioning that this is not the only literature on Lorentzian trace formulae rather the first step towards this direction has been taken up for the d'Alembertian on a spatially compact globally hyperbolic stationary spacetime by 
Strohmaier and Zelditch~\cite{Strohmaier_AdvMath_2021}. 
Therefore, this article can be viewed as the Lorentzian generalisation of Sandoval's work,  propounding the framework of Strohmaier-Zelditch into a bundle setting. 
%
%
%
%
%
%
%
%
%
%
\subsection{Primary results} 
\label{sec: result}
Let $\sE \to \sM$ be a vector bundle over a $d \geq 2$ dimensional spatially compact globally hyperbolic stationary spacetime $\SST$. 
That is, the spacetime manifold $\sM$ is homeomorphic to the product manifold $\R \times \varSigma$ where $\varSigma \subset M$ is a spacelike Cauchy hypersurface and $\spacetime$ admits a \textit{complete} timelike Killing vector field $Z$ whose flow is denoted by $\R \times \sM \ni (s, x) \mapsto \varXi (s, x) =: \varXi_{s} (x) \in \sM$.  

Recall, the Lorentzian volume element on $M$ paves way the natural identification between a section and a half-density-valued section of $\sE$. 
We will always work with the latter yet suppress the canonical half-density bundle notationally just for brevity. 
The Dirac-type operator $D$ acting on (half-density-valued-)smooth sections of $\sE$ is a first-order linear differential operator whose principal symbol $\symb{D}$ satisfies the Clifford relation
\begin{equation}
	\symb{D} \xxi^{2} = \fg_{x}^{-1} (\xi, \xi) \, \one_{\End{\sE}_{x}}, \quad \forall \xxi \in C^{\infty} (\sM; \coTansM),  
\end{equation}
where $\fg^{-1}$ represents the spacetime metric on the cotangent bundle $\coTanM$ and $\one_{\End{\sE}}$ is the ($1$-density-valued) identity endomorphism on $\sE$. 

\begin{assumption} \label{asp: trace_formula}
	We consider a Dirac operator subjected to the following assumptions:  
	\begin{enumerate}[label=(\alph*)] 
		\item 
		$\sE \to M$ is endowed with a sesquilinear form $(\cdot|\cdot)$ invariant under the Killing flow $\varXi_{s}^{*}$ such that $D$ is symmetric;  
		\label{asp: Direc_op_symmetric_SST}
		\item 
		Given an arbitrary but fixed future-directed unit normal covector field $(\cdot, \zeta)$ on $\sM$ along $\varSigma$, 
		\begin{equation} \label{eq: hermitian_form_Dirac_type_op}
			\langle \cdot | \cdot \rangle := \, \big( \symb{D} (\cdot, \zeta ) \cdot \big| \cdot \big)   
		\end{equation}
		is a fibrewise hermitian form on the bundle of Clifford modules $\big( \sE \to \sM, (\cdot|\cdot), \symb{D} \big)$; 
		\label{asp: hermitian_form_Dirac_op_SST}
		\item 
		$D$ commutes with the induced Killing flow $\varXi_{s}^{*}$ on $\sE$ for all $s \in \R$.
		\label{asp: commutator_Dirac_op_Killing_flow}
	\end{enumerate} 
\end{assumption}

We note that the Clifford module bundle $\big( \sE \to M, (\cdot|\cdot), \symb{D} \big)$ is naturally furnished with the unique Weitzenb\"{o}ck connection $\nabla$ (see Section~\ref{sec: Dirac_op}) induced by $D$ and $\fg$. 
This connection induces a fibrewise canonical isomorphism $\hat{\varXi}_{s} : \sE_{x} \to (\varXi_{s}^{*} \sE)_{\varXi_{-s} (x)}$ along the integral curve $\varXi_{s} (\cdot)$ of $Z$. 
In other words, the following diagram commutes fibrewise: 

%
%
%
\begin{center}
	\begin{tikzpicture}
		\node (a) at (0, 0) {$\sM$};
		\node (b) at (3, 0) {$\sM$};
		\draw[->] (a) -- (b); 
		\node (c) at (0, 1.5) {$\varXi_{s}^{*} \sE$}; 
		\node (d) at (3, 1.5) {$\sE$}; 
		\draw[->] (c) -- (d); 
		\draw[->] (c) -- (a); 
		\draw[->] (d) -- (b); 
		\node[above] at (1.5, 0) {$\varXi_{s}$}; 
		\node[above] at (1.5, 1.5) {$\hat{\varXi}_{-s}$}; 
	\end{tikzpicture}
	\captionof{figure}{Fibrewise lift $\hat{\varXi}_{-s}$ of spacetime isometry $\varXi_{s}$.}
	\label{fig: bundle_lift_spacetime_isometry}
\end{center}
%
%
%

Let $\pounds_{\! Z}$ be the Lie derivative on $\sE$ (given by~\eqref{eq: def_Lie_deri}) and one sets 
\begin{equation} \label{eq: def_L}
	L := - \ri \pounds_{\! Z} : \secsME \to \secsME
\end{equation}
so that the induced Killing flow on $\secsME$ is expressed as $\varXi_{t}^{*} = \re^{\ri t L}$.  
By Assumption~\ref{asp: trace_formula}~\ref{asp: commutator_Dirac_op_Killing_flow}, 
\begin{equation}
	LD = DL. 
\end{equation}
Thus, on $\ker D$, the eigensections $\uppsi_{n}$ of $L$ are the joint eigensections: 
\begin{equation}
	D \uppsi_{n} = 0, \qquad L \uppsi_{n} = \lambda_{n} \uppsi_{n}. 
\end{equation}
We equip $\ker D$ with the hermitian inner product $\scalarProdTwo{\cdot}{\cdot}$ to have the Hilbert space $(\ker D, \scalarProdTwo{\cdot}{\cdot})$. 
 
%
%
%
\begin{theorem} \label{thm: spectrum_L}
	Under Assumption~\ref{asp: trace_formula}, the spectrum of $L$~\eqref{eq: def_L} on the Hilbert space $\sH := (\ker D, \langle \cdot | \cdot \rangle)$ is purely discrete and comprises infinitely many real eigenvalues that grow polynomially and accumulate at $\pm \infty$. 
\end{theorem}
%
%
%

Thus, we can restrict our attention entirely to smooth sections $C^{\infty} (\sM; \sE)$ of $\sE$ owing to the elliptic regularity and $\tr_{\sH} \re^{\ri t L}$ makes sense as a distribution with the identification; cf.~\eqref{eq: trace_time_evolution_op_tr_Killing_flow} 
\begin{equation} \label{eq: Tr_Cauchy_evolution_op_Dirac_Tr_Killing_flow}
	\Tr U_{t} = \tr_{\sH} \varXi_{t}^{*},   
\end{equation}
where $U_{t}$ is the time evolution operator of the Dirac equation~\eqref{eq: Dirac_eq}. 

Since the classical dynamics is governed by the principal symbol $\symb{D}$ of $D$, its characteristic set $\Char D := \{ \xxiNot \in \coTansM | \xi^{0} \neq 0, \symb{D} \xxiNot = 0 \}$ can be viewed as a classical limit of $\ker D$ as per se geometric quantisation. 
Thus our classical phase space is the lightcone bundle $\coLightBun \to \sM$ over $\SST$ where the Hamiltonian 
\begin{equation} \label{eq: def_Hamiltonian_metric}
	H_{\fg} : C^{\infty}(\sM; \coTansM) \to \R, ~ \xxi \mapsto H_{\fg} \xxi := \frac{1}{2} \fg_{x}^{-1} (\xi, \xi) 
\end{equation}
induced by $\fg^{-1}$ vanishes and the reduced phase space is the conic symplectic manifold $\cN$ of scaled-lightlike geodesic strips~\cite{Penrose_1972} 
(see also~\cite{Khesin_AdvMath_2009}). 
Being a globally hyperbolic spacetime, $\sM$ does not admit any closed timelike geodesic. 
Therefore, the notion of Lorentzian analogue of periodic trajectories is defined by means of the isometry $\varXi_{t}$ induced (reduced) symplectic flow $\varXi_{t}^{\cN}$ on $\cN$, whose Hamiltonian is given by~\cite[Lem. 1.1]{Strohmaier_AdvMath_2021} 
\begin{equation} \label{eq: def_Hamiltonian_scaled_lightlike_geodesic}
	H : \cN \to \R, ~\gamma \mapsto H (\gamma) := \fg \Big( \frac{\rd c}{\rd s}, Z \Big), 
	\quad \forall s \in \R, 
\end{equation} 
where $\R \ni s \mapsto c (s) \in M$ is any lightlike geodesic on $\sM$ and the value $\fg (\nicefrac{\rd c}{\rd s}, Z)$ is independent of the cotangent lift $\gamma$ of $c$. 
This Hamiltonian is positive as $Z$ is timelike. 
For any $\tE \in \R_{+}$, we denote the constant $\tE$-energy surface by  
\begin{equation} \label{eq: def_cst_energy_surface_Hamiltonian_scaled_lightlike_geodesic}
	\cN_{\tE} := \{ \gamma \in \cN \,|\, H (\gamma) = \tE \},  
\end{equation}
and then the set of periods resp. periodic lightlike geodesic strips of $\varXi_{s}^{\cN}$ are given by 
(see e.g.~\cite[(6)]{Strohmaier_AdvMath_2021}) 
\begin{subequations}
	\begin{eqnarray}
		\cP 
		& := & 
		\{ T \in \R_{+} \,|\, \exists \gamma \in \cN : \varXi_{T}^{\cN} (\gamma) = \gamma \}, 
		\label{eq: def_period_Killing_flow}
		\\  
		\cP_{T} 
		& := & 
		\{ \gamma \in \cN \,|\, \varXi_{T}^{\cN} (\gamma) = \gamma \}. 
		\label{eq: def_periodic_pt_Killing_flow}
	\end{eqnarray}
\end{subequations}
Recall, the set of all lengths of periodic geodesics on a manifold counted with multiplicities is called the length spectrum of the manifold. 
Our first finding states that $\Tr U_{t}$ determines the Lorentzian length spectrum of $\cN$ in the sense described below. 

%
%
%
\begin{proposition} \label{prop: length_spectrum}
	Let $\sE \to \sM$ be a vector bundle over a spatially compact globally hyperbolic stationary spacetime $\SST$, endowed with a sesquilinear form $(\cdot|\cdot)$ and $D$ a Dirac-type operator on $\sE$ whose principal symbol is $\symb{D}$ so that $\big( \sE \to \sM, (\cdot|\cdot), \symb{D} \big)$ is a bundle of Clifford modules over $\SST$.   
	If $U_{t}$ is the time evolution operator of $D$ then under Assumption~\ref{asp: trace_formula}, $\Tr U_{t}$ is a distribution on $\R$ and its singular support 
	\begin{equation}
		\singsupp{\Tr U_{t}} \subset \{0\} \cup \cP,  
	\end{equation}
	where $\cP$ is the set of periods of the induced Killing flow $\varXi_{T}^{\cN}$ on the manifold of scaled-lightlike geodesic strips $\cN$, given by~\eqref{eq: def_period_Killing_flow}.   
\end{proposition}
%
%
%

Lagrangian distributions~\cite{Hoermander_ActaMath_1971, Duistermaat_ActaMath_1972} 
(see also the treatise~\cite{  Hoermander_Springer_2009})  
offer the elegant characterisation of $\Tr U_{t}$, as stated below.  

%
%
%
\begin{theorem} \label{thm: trace_formula_t_zero}
	As in the set-up of Proposition~\ref{prop: length_spectrum}, $\Tr U_{T}$ is the Lagrangian distribution $I^{d - 7/4} (\R, \Lambda_{T})$ where $d := \dim \sM$ and the Lagrangian manifold $\Lambda_{T}$ is given by  
	\begin{equation} \label{eq: def_canonical_relation_Tr_U_t}
		\Lambda_{T} := \{ (T, \tau) \in \R \times \R_{-} \,|\, \exists \gamma \in \cN : \varXi_{T}^{\cN} (\gamma) = \gamma, \tau = - H (\gamma) \},   
	\end{equation}
	where the Hamiltonian $H$ on $\cN$ is given by~\eqref{eq: def_Hamiltonian_scaled_lightlike_geodesic}. 
	Furthermore,    
	\begin{equation}
		(\Tr U_{0}) (t) = u_{0} (t) + v_{0} (t), 
	\end{equation}
	where $v_{0}$ is a distribution on $\R$ which is smooth in the vicinity of $t=0$ and $u_{0}$ is a Lagrangian distribution admitting the following singularity expansion around $t=0$:   
	\begin{eqnarray}
		u_{0} (t) 
		& \sim & 
		r (d-1) \frac{\vol (\cN_{H \leq 1})}{(2 \pi)^{d-1}} \mu_{d-1} (t) + \mathrm{c}_{2} \, \mu_{d-2} (t) + \ldots, 
		\nonumber \\ 
		r 
		& := & 
		\rk{\sE} \int_{\varSigma} \fg_{x}^{-1} (\eta, \zeta) \, \dVolh (x), 
		\nonumber \\ 
		\mu_{d-k} (t) 
		& := & 
		\int_{\R_{\geq 0}} \re^{- \ri t \tau} \tau^{d-1-k} \rd \tau, \quad k = 1, 2, \ldots.   
	\end{eqnarray}
	Here, $\mu_{d-k}$ is a distribution on $\R$ given by the preceding oscillatory integral, $\eta$ is any lightlike covector on $\sM$ and $\zeta$ as in Assumption~\ref{asp: trace_formula}, both restricted to a Cauchy hypersurface $\varSigma$ of $\sM$, $\mathrm{c}_{2}$ is some constant (Dirac-wave trace invariant), and $\vol (\cN_{H \leq 1})$ is the volume of $\cN_{H \leq 1}$. 
\end{theorem}
%
%
%

We remark that $\SST$ can be considered as a (globally hyperbolic) standard stationary spacetime as $Z$ is complete~\cite[Thm. 2.3]{Candela_AdvMath_2008}. 
Since there exists a coordinate neighbourhood $\big( U, (t, x^{i}) \big)$ for any spacetime point on which $Z = \partial_{t}$, we can choose $t$ as the Killing-time coordinate so that the spacetime metric does not depend on $t$. 
More precisely, on $U$ the spacetime metric splits non-uniquely as 
\begin{equation}\label{eq: def_spacetime_metric_SSST}
	\fg := \upbeta^{2} \, \rd t^{2} - \fh_{ij} (\rd x^{i} + \upalpha^{i} \, \rd t) (\rd x^{j} + \upalpha^{j} \, \rd t), 
	\quad i,j = 2, \ldots, d,     
\end{equation}
where $\fh$ is a Riemannian metric on the (compact) Cauchy hypersurface $\varSigma$, $\upalpha$ is the shift vector field and $\upbeta$ is the lapse function. 
Both $\upalpha$ and $\upbeta$ are independent of $t$ but depend on $\varSigma$. 
In the preceding form of the spacetime metric, the volume of $\cN_{H \leq 1}$ is given by~\cite[$(15)$]{Strohmaier_AdvMath_2021}
\begin{equation} \label{eq: vol_energy_surface}
	\vol (\cN_{H \leq 1}) = \vol (\mathbb{B}^{d-1}) \int_{\varSigma} \upbeta (x) \big( \upbeta^{2} (x) - \fh_{x} (\upalpha, \upalpha) \big)^{-d/2} \dVolh (x),  
\end{equation}
where $\vol (\mathbb{B}^{d-1})$ is the volume of the unit ball $\mathbb{B}^{d-1} \subset \R^{d-1}$. 

Since $L$ has discrete eigenvalues, let us introduce its eigenvalue counting  
function\footnote{One   
	can define it with negative eigenvalues as well.
} 
\begin{equation}
	\mathsf{N} (\lambda) := \# \{ n \in \N \,|\, 0 \leq \lambda_{n} \leq \lambda \}. 
\end{equation}
%
%
%
\begin{corollary}[Weyl law] \label{cor: Weyl_law}
	As in the terminologies of Theorem~\ref{thm: spectrum_L} and Theorem~\ref{thm: trace_formula_t_zero}, the Weyl eigenvalue counting function of $L$ has the asymptotics 
	\begin{equation}
		\fN (\lambda) = \Big( \frac{\lambda}{2 \pi} \Big)^{d-1} r \, \vol (\cN_{H \leq 1}) + O (\lambda^{d-2}), 
		\quad \textrm{as} \quad \lambda \to \infty. 
	\end{equation}
\end{corollary}
%
%
%

In order to describe $\Tr U_{T}$ for the non-trivial periods $T \neq 0$, we require the following concepts from the theory of pseudodifferential operators. 
Let $\sE \to M$ be a smooth complex vector bundle over a manifold $M$ and $P$ a pseudodifferential operator of order at most $m \in \R$ acting on compactly supported sections of $\sE$ valued in the half-density bundle $\varOmega^{\nicefrac{1}{2}} \to M$ which will be put down just for notational simplicity. 
$P$ is said to be non-characteristic at some element $\xxi$ in the punctured cotangent bundle $\dotCoTanM$ of $M$ if $\tilde{p} \circ \symb{P} - \mathtt{1}$ belongs to the symbol class of order $-1$ in a conic neighbourhood of $\xxi$ for some symbol $\tilde{p}$ of order $-m$. 
The set of all characteristic points of $P$ is denoted by $\Char{P}$ 
(see e.g.~\cite[Def. 18.1.25]{Hoermander_Springer_2007}). 
We remark that the subprincipal $\subSymb{P}$ of $P$ is an invariantly defined $\Hom{\sE, \sE}$-valued homogeneous function on $\dotCoTanM$~\cite[$(5.2.8)$]{Duistermaat_ActaMath_1972}  
(see also~\cite[Thm. 18.1.33]{Hoermander_Springer_2007}).  
%
%
%
\begin{definition} \label{def: P_compatible_connection}
	Let $(\sE, F \to M; \connectionE, \nabla^{\ms F})$ be two smooth complex vector bundles over a smooth manifold $M$, equipped with their connections $\connectionE, \nabla^{\ms F}$, and $P$ a pseudodifferential operator from half-density-valued compactly supported smooth sections of $\sE$ to half-density-valued smooth sections of $F$. 
	A connection $\nabla^{\ms \sE \otimes F}$ on the tensor product bundle $\sE \otimes F \to M$ will be called $P$-compatible if and only if ~\cite{Islam}  
	\begin{equation}
		\mathsf{\Gamma} \big( (\rd_{x} \pi) X_{p} \big) \xxi = \ri \subSymb{P} \xxi, 
		\quad 
		\forall \xxi \in \Char{P},   
	\end{equation}
	where $\subSymb{P}$ resp. $\Char{P}$ are the subprincipal symbol resp. the characteristic set of $P$, $\mathsf{\Gamma}$ is the connection $1$-form of $\nabla^{\ms \sE \otimes F}$ and $\pi: \dotCoTanM \to M$ is the punctured cotangent bundle. 
	In other words, $\nabla^{\ms \sE \otimes F}$ induces the covariant derivative 
	\begin{equation}
		\nabla_{X_{p}}^{\pi^{*} \Hom{\sE, F}} = X_{p} + \mathsf{\Gamma} \big( (\rd \pi) X_{p} \big) 
	\end{equation}
	on the bundle $\pi^{*} \Hom{\sE, F} \to \dotCoTanM$ along the Hamiltonian vector field $X_{p}$ generated by the principal symbol $p$ of $P$.  
\end{definition}
%
%
%

As mentioned above, the Clifford module bundle $\big( \sE \to M, (\cdot|\cdot), \symb{D} \big)$ is naturally furnished with the unique Weitzenb\"{o}ck connection $\nabla$ (see Section~\ref{sec: Dirac_op}). 
Let $\connectionEndPiE_{X_{\fg}}$ be the $D^{2}$-compatible covariant derivative with respect to the geodesic vector field $X_{\fg}$ on $\pi^{*} \Hom{\sE, \sE} \to \dotCoTansM$. 
By $\Hol := C_{\mathrm{Hol}}^{\infty} \big( \dotCoTansM; \pi^{*} \Hom{\sE, \sE} \big)$ we will denote the set of all sections of $\pi^{*} \Hom{\sE, \sE}$ invariant under the holonomy group of the  parallel transporter $\cT$ with respect to $\connectionEndPiE_{X_{\fg}}$.  

We suppose that $\rP_{\gamma}$ is the linearised Poincar\'{e} map of a periodic geodesic $\gamma$ 
(see e.g.~\cite[Sec. 7.1]{Abraham_AMS_1978},~\cite{Strohmaier_AdvMath_2021}). 
Then $\gamma$ is called non-degenerate if $1$ is not an eigenvalue of $\rP_{\gamma}$. 

With all these devises, the precise characterisation of $\Tr U_{T \neq 0}$ can be stated as follows. 

%
%
%
\begin{theorem} \label{thm: trace_formula_t_T}
	As in the set-ups of Proposition~\ref{prop: length_spectrum} and Theorem~\ref{thm: trace_formula_t_zero}: if the periods $T$ are discrete and the set $\cP_{T}$~\eqref{eq: def_periodic_pt_Killing_flow} of periodic lightlike geodesic strips of $\varXi_{T}^{\cN}$ is a finite union of non-degenerate periodic orbits $\gamma$ then  
	\begin{equation}
		(\Tr U_{T}) (t) = \sum_{\gamma \in \cP_{T}} u_{\gamma} (t) + v_{\ms T} (t), 
	\end{equation}
	where $v_{\ms T}$ is a distribution on $\R$ that is smooth in the vicinity of $t=T$ and $u_{\gamma} (t)$'s are Lagrangian distributions having singularities at $t=T_{\gamma}$ with the asymptotic expansion  
	\begin{eqnarray} \label{eq: asymptotic_expansion_Lagrangian_dist_T}
		\lim_{t \to T_{\gamma}} (t - T_{\gamma}) u_{\gamma} (t) 
		& \sim & 
		\frac{1}{2 \pi}  
		\int_{ \kF_{T_{\gamma}} } \tr \Big( \symb{D} (\gamma) \, \cT_{\varXi_{\ms T}^{\ms \cN} (\gamma)} \symb{D} (x, \zeta) \Big) \frac{\re^{- \ri \pi \km (\gamma) / 2} |\rd T_{\gamma}|}{\sqrt{| \det (I - \rP_{\gamma}) |}} \nu
		+ \ldots, 
		\nonumber \\ 
		\nu 
		& := & 
		\int_{\R_{\geq 0}} \re^{- \ri (t - T) \tau} \Big( \frac{\tau}{2 \pi} \Big)^{d-2} \rd \tau, 
	\end{eqnarray}
	where $\km (\gamma)$ resp. $\rP_{\gamma}$ are the Maslov index resp. the Poincar\'{e} map of $\gamma$, $\mathcal{T}_{\gamma}$ is an element of the holonomy group $\Hol_{\gamma}$ with respect to the $D^{2}$-compatible Weitzenb\"{o}ck connection at base point $\gamma \in \cN$ whose projection on $\sM$ is $x$, and 
	\begin{equation} \label{eq: def_F_T}
		\kF_{T} := \left\{ \gamma \in \cN \,|\, \tau = - H (\gamma), \varXi_{T}^{\cN} (\gamma) = \gamma \right\} 
	\end{equation}
	is the fixed point set of $\varXi_{T}^{\cN}$.  
\end{theorem}
%
%
%

At this point we would like to comment on the issue of parallel transporter raised at the beginning of this article and reproduce  
Sandoval's work~\cite[Thm. 2.8]{Sandoval_CPDE_1999} 
as a special case of this theorem. 
We set $\upalpha = 0$ and $\upbeta = 1$ in~\eqref{eq: def_spacetime_metric_SSST} so that the Dirac equation~\eqref{eq: Dirac_eq} on the ultrastatic spacetime becomes 
\begin{equation} \label{eq: Dirac_Hamiltonian_ultrastatic}
	- \ri \partial_{t} u = H_{D} u, 
	\quad 
	H_{D} := \fc (\rd t) \big( \hat{D} - U \big), 
\end{equation}
where $H_{D}$ is the Dirac Hamiltonian, $\fc$ is the Clifford multiplication (see~\eqref{eq: def_Clifford_multiplication}) and $U$ is the potential (see~\eqref{eq: Dirac_op}) allowed by the most general Dirac-type operator. 
Sandoval plumped for $U = 0$ and used the parallel transporter $\tilde{\cT}$ corresponding to $\tilde{D}$ instead of $\hat{D}$ and hence $\tilde{T}$ together with a suitable average of $W := \hat{D} - \tilde{D}$ showed up under the $\tr$ in~\eqref{eq: asymptotic_expansion_Lagrangian_dist_T} in her result. 
However, no such artificial choices are necessary in our formulation at all, rather $\cT$ naturally induces the parallel transport $\hat{\cT}$ with respect to $\hat{D}$ for the preceding choices.      

In the future, several generalisations have been planned. 
For instance, we intend to address the spectral asymptotics on stationary black holes and explore the semi-classical regime.   
We also wish to extend the study for Hodge-d'Alambertians and connect with interesting applications on relativistic quantum chaos on curved spacetimes. 
%
%
%
%
%
%
%
%
%
%
\subsection{Convention} 
\label{sec: convention}
Throughout the article, a vector bundle $\sE \to M$ means a smooth complex vector bundle over a Hausdorff, second countable $d \in \N$-dimensional smooth manifold $M$. 
We use the notation $\dot{\sE}$ to symbolise the zero-section removed part of $\sE$. 
Occasionally, $\sE$ is endowed with a non-degenerate sesquilinear form $(\cdot|\cdot)$ which is assumed to be anti-linear in its first argument. 
By a hermitian form $\langle \cdot | \cdot \rangle$ on $\sE$ we mean a positive-definite $(\cdot|\cdot)$.
A Lorentzian manifold and its special case - a globally hyperbolic manifold, both are denoted by $\spacetime$ with metric signature $+ - \ldots -$,  and $d := \dim \sM \geq 2$.  
Unless mentioned otherwise, $(\varSigma, \iota)$ represents an immersed (resp. embedded)  submanifold of a manifold $M$ (resp. a globally hyperbolic manifold $\spacetime$). 
Then $E_{\varSigma}$ resp. $\coTanM_{\varSigma} \to \varSigma$ represent the pullback bundles of $\sE$ resp. the cotangent bundle $\coTanM$ via the immersion $\iota$.  

Let $\halfDen \to M$ be the bundle of half-densities over $M$. 
We denote the vector spaces of smooth and compactly supported smooth half-densities on $\sE$ by $C^{\infty} (M; \sE \otimes \halfDen)$ and $C_{\rc}^{\infty} (M; \sE \otimes \halfDen)$, and endow with the Fr\'{e}chet space and the inductive limit topologies, respectively. 
The vector spaces of distributional half-densities $\cD' (M; \sE \otimes \halfDen) := \big( C_{\rc}^{\infty} (M; \sE^{*} \otimes \halfDen) \big)'$ and compactly supported distributional half-densities $\mathcal{E}' (M; \sE \otimes \halfDen) := \big( C^{\infty} (M; \sE^{*} \otimes \halfDen) \big)'$ on $\sE$ are defined by the topological duals of $C_{\rc}^{\infty} (M; \sE^{*} \otimes \halfDen), C^{\infty} (M; \sE^{*} \otimes \halfDen)$, which are equipped with the weak $*$-topologies induced by the topologies of $C_{\rc}^{\infty} (M; \sE^{*} \otimes \halfDen)$ resp. $C^{\infty} (M; \sE^{*} \otimes \halfDen)$, where $\sE^{*} \to M$ is the dual bundle of $\sE$.  
Unless stated otherwise, we have notationally suppressed $\halfDen$ on a generic manifold $M$ for brevity, whereas, on a Lorentzian manifold $\spacetime$ we use the natural Lorentzian volume element $\dVolg$ to identify $\secsME = C^{\infty} (\sM; \sE \otimes \halfDen), \dualComSecsME = \cD' (\sM; \sE \otimes \halfDen)$. 
Additionally, one defines the space $C_{\mathrm{sc}}^{\infty} (\sM; \sE)$ of spatially compact smooth sections of $\sE$ as the set of all $u \in \secsME$ for which there exists a compact subset $K$ of $\sM$ such that $\supp{u} \subset J (K)$ where $J (K) := J^{+} (K) \cup J^{-} (K)$ and $J^{\pm} (K)$ are the causal future(past) of $K$; as a vector space $C_{\mathrm{sc}}^{\infty} (\sM; \sE) \subset \secsME$~\cite[Not. 3.4.5]{Baer_EMS_2007}.   
Echoing this spirit $C^{\infty} (M)$ denotes the set of all complex-valued smooth functions on $M$. 

The principal symbol of a normally hyperbolic operator $\square$ on $E$ is given by 
\begin{equation}
	\symb{\square} \xxi := \fg_{x}^{-1} (\xi, \xi) \, \one_{\End \sE_{x}}, \quad \forall \xxi \in C^{\infty} (\sM, \coTanM).  
\end{equation}
In other words, in any coordinate frame $(\partial / \partial x^{i})$ on a chart $\big( U, (x^{i}) \big)$ of $M$, after bundle trivialisations:   
\begin{equation}
	\square |_{U} = - \fg^{ij} \parDeri{x^{i}}{} \parDeri{x^{j}}{} + \textrm{lower order terms}, 
	\quad i, j = 1, \ldots, d   
\end{equation}
where Einstein's summation convention has been used. 

The Hamiltonian~\eqref{eq: def_Hamiltonian_metric} of the geodesic flow $\varPhi_{s}$ is defined with a $1/2$ prefactor so that the relativistic Hamiltonian flow is identical to the geodesic flow as in ~\cite[Rem. 7.2]{Strohmaier_AdvMath_2021}.    

In this article, only the polyhomogeneous symbol class $S^{m} (\cdot)$ has been used and we follow the  
convention\footnote{In 
	particular, our convention is identical to that used by  
	Strohmaier-Zelditch~\cite[Sec. 7]{Strohmaier_AdvMath_2021}  
	but slightly different from   
	Duistermaat-Guillemin~\cite[Sec. 7]{Duistermaat_Inventmath_1975} 
	and Sandoval~\cite[Sec. 2.1]{Sandoval_CPDE_1999}.} 
as in 
H\"{o}rmander's treatises~\cite{Hoermander_Springer_2007, Hoermander_Springer_2009} 
for Fourier integral operators.   
Thus, a Fourier integral operator $A : C_{\mathrm{c}}^{\infty} (N; F) \to \cD' (M; E)$ of order at most $m \in \R$ associated with a homogeneous canonical relation 
(see e.g.~\cite[Def. 21.2.12]{Hoermander_Springer_2007}) 
$C \subset \dotCoTanM \times \dotCoTanN$ which is closed in $\dotCoTanMN$, is a continuous linear operator whose Schwartz kernel $\fA$ is an element in the space $I^{m} \big( M \times N, C'; \Hom{F, E} \big)$ of Lagrangian distributions 
(see e.g.~\cite[Def. 25.2.1]{Hoermander_Springer_2009}), 
where $F \to N$ is a smooth complex vector bundle over a smooth manifold $N$.  
Locally this means that $\fA = (\fA^{r}_{k})_{\rk \sE \times \rk F}$ can be identified with a matrix of scalar Lagrangian distributions $\fA^{r}_{k}$ on respective Euclidean spaces and modulo smoothing kernels each $\fA^{r}_{k}$ is of the form 
(see e.g.~\cite[Prop. 25.1.5$'$]{Hoermander_Springer_2009}) 
\begin{equation} \label{eq: FIO_local}
	\fA_{k}^{r} = (2 \pi)^{- (\dim M + \dim N + 2n - 2e) / 4} \int_{\Rn} \rd \theta \, \re^{\ri \varphi (x, y; \theta)} \, \fa_{k}^{r} \xytheta,   
\end{equation}
where $\varphi$ is a clean phase function 
(see e.g.~\cite[Def. 21.2.15]{Hoermander_Springer_2007}) 
with excess $e$ and $\fa_{k}^{r}$ is a symbol of order $m + (\dim M + \dim N - 2n - 2e)/4$ having support in the interior of a sufficiently small conic neighbourhood of the fibre-critical manifold $\sC$ of $\varphi$ contained in the domain of definition of $\varphi$. 
As per se the chosen symbol class, each $\fa^{r}_{k}$ can be expressed by the asymptotic series
(see e.g.~\cite[Prop. 18.1.3, Def. 18.1.5]{Hoermander_Springer_2007}) 
\begin{equation} \label{eq: def_polyhomogeneous_symbol}
	\fa_{k}^{r} \sim a_{k}^{r} + \sum_{l \in \N} a_{k; l}^{r},   
\end{equation}
for large $\theta$, where $a^{r}_{k}$ is a function of $\xytheta$ whose degree of homogeneity is the same as that of $\fa^{r}_{k}$ and the degrees of $a^{r}_{k; l}$'s are dropped by a factor of $l$'s from the top-degree when $|\theta| \geq 1$. 
In the $|\theta| < 1$ regime, one multiplies the preceding expression by a cutoff function vanishing identically near $\theta = 0$ and which becomes identity whenever $|\theta| \geq 1$. 
The principal symbol $\symb{\fA}$ of $\fA$ is locally given by the matrix of elements 
(see e.g.~\cite[Prop. 25.1.5]{Hoermander_Springer_2009}) 
\begin{equation} \label{eq: symbol_FIO_nondegenerate}
	\sigma_{A_{k}^{r}} \xxiyeta = a_{k}^{r} \dfrac{ |\rd \xi|^{\nicefrac{1}{2}} |\rd \eta|^{\nicefrac{1}{2}} }{\sqrt{|\det (\mathrm{Hess} \, \varphi)|}} \re^{\ri \pi \, \sgn (\mathrm{Hess} \, \varphi) / 4}, 
\end{equation}
when $\varphi$ is non-degenerate and here $\mathrm{Hess}$ denotes the Hessian; for a clean phase function the expression is a bit complicated and available 
in~\cite{Duistermaat_Inventmath_1975} 
(see also~\cite[(25.1.5$'$)]{Hoermander_Springer_2009}). 
In the context of principal symbols, we have used $\equiv$ to mean modulo the Keller-Maslov contribution $\re^{\ri \pi \, \sgn (\mathrm{Hess} \, \varphi) / 4}$, i.e., 
\begin{equation}
	\sigma_{A_{k}^{r}} \xxiyeta \equiv a_{k}^{r} \dfrac{ \sqrt{|\rd \xi| |\rd \eta|} }{\sqrt{|\det (\mathrm{Hess} \, \varphi)|}}.
\end{equation}
The Fourier transform $\cF (u)$ of any $u \in L^{1} (\Rd, \rd x)$ is   
\begin{equation} \label{eq: def_Fourier_transformation} 
	\mathcal{F} (u) := \int_{\Rd} \re^{ - \ri x \cdot \theta} u \, \rd x  
	~\quad~\textrm{and}~\quad  
	u = \frac{1}{(2 \pi)^{d}} \int_{\Rd} \re^{ \ri x \cdot \theta} \cF (u) \, \rd \theta,    
\end{equation}
whenever $\cF (u)$ is integrable and here $\cdot$ is either the Euclidean or the Minkowski inner product depending on the context. 

We have used the symbols $[\cdot, \cdot]_{\pm}$ for the (anti-)commutator brackets. 
%
%
%
%
%
%
%
%
%
%
\subsection{Proof strategy and novelty} 
\label{sec: proof_strategy}
We divide the description in several steps to give a panorama view. 
%
%
%
%
%
%
%
%
%
%
\subsubsection{$U_{t}$ as a Fourier integral operator} 
This pivotal idea was originally due to 
Duistermaat and Guillemin~\cite{Duistermaat_Inventmath_1975} 
who worked it out (modulo smoothing operators) in the context of scalar half-wave operators on an ultrastatic spacetime. 
We, however, have not followed their approach directly, instead have expressed $U_{t}$ in terms of the Killing flow $\varXi_{t}^{*}$ and the causal propagator (see~\eqref{eq: def_causal_propagator_Dirac}) $G$ of $D$ by propounding 
Strohmaier and Zelditch's work on d'Alembertian~\cite{Strohmaier_AdvMath_2021}. 
More precisely, there exists unique advanced and retarded Green's operators for $D$ owing to the global hyperbolicity of $\spacetime$ and hence $G$ together with the restriction operator $\iota_{\varSigma}^{*}$ (see Appendix~\ref{sec: restriction_op}) pave the way to construct the Cauchy restriction operator $\cR$ as described in Section~\ref{sec: Cauchy_problem_Dirac}. 
By Assumption~\ref{asp: trace_formula}~\ref{asp: commutator_Dirac_op_Killing_flow}, the time flow determines the time evolution of Cauchy data. 
The combination of these facts allows us to express $U_{t}$ as a Lagrangian distribution as inscribed in Lemma~\ref{lem: symbol_Tr_U_0}. 
The preceding notions are depicted schematically in Figure~\ref{fig: time_evolution_causal_propagator_restriction_op_Killing_flow}. 
In order to describe $U_{t}$ as a Fourier integral operator one is required to obtain such a description of $G$ and $\iota_{\ms \varSigma}^{*}$. 
Both are well-known for the scalar case and the bundle generalisation is straightforward for the latter. 
But the former demands an intricate treatment due to possible bundle curvature, which has been spelled out in Lemma~\ref{lem: causal_propagator_Dirac}. 
In fact, this is precisely from where the expedient choice (as discussed in the introduction) in Sandoval's work stems. 
We have computed $\symb{\fG}$ in Lemma~\ref{lem: causal_propagator_Dirac} deploying the $D^{2}$-compatible Weitzenb\"{o}ck connection in an intrinsically geometric fashion. 
As a consequence, it solves the relevant leading order transport equation concisely from where the holonomy group shows up elegantly. 
Furthermore, it closes the ad-hoc consideration in Sandoval's analysis as explained after Theorem~\ref{thm: trace_formula_t_T}.  

\begin{center}
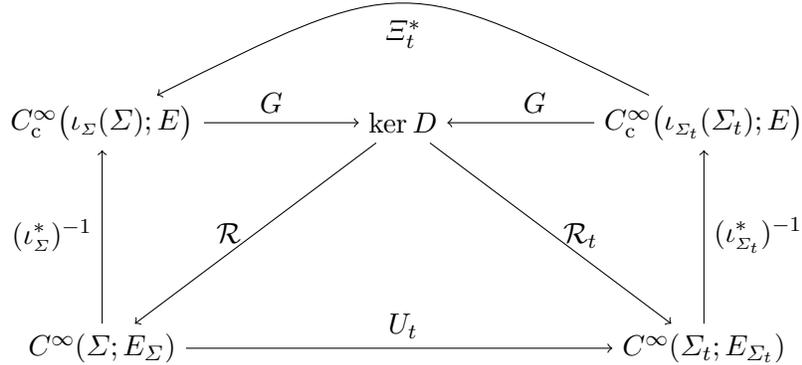

	\begin{tikzpicture}
		\node (a) at (0, 0) {$C^{\infty} (\varSigma; \sE_{\varSigma})$}; 
		\node (b) at (8, 0) {$C^{\infty} (\varSigma_{t}; \sE_{\varSigma_{t}})$}; 
		\node (c) at (0, 3) {$C_{\mathrm{c}}^{\infty} \big( \iota_{\ms \varSigma} (\varSigma); \sE \big)$}; 
		\node (d) at (8, 3) {$C_{\mathrm{c}}^{\infty} \big( \iota_{\ms \varSigma_{t}} (\varSigma_{t}); \sE \big)$}; 
		\node (e) at (4,3) {$\ker D$}; 
		\draw[->] (a) -- (b); 
		\draw[->] (a) -- (c); 
		\draw[->] (b) -- (d); 
		\draw[->] (e) -- (a); 
		\draw[->] (e) -- (b); 
		\draw[->] (c) -- (e); 
		\draw[->] (d) -- (e); 
		\draw[->] (d) .. controls (4,5) .. (c);
		\node[above] at (4,0) {$U_{t}$}; 
		\node[left] at (2, 1.5) {$\cR_{}$}; 
		\node[right] at (6, 1.5) {$\cR_{t}$}; 
		\node[left] at (0, 1.5) {$(\iota_{\ms \varSigma}^{*})^{-1}$}; 
		\node[right] at (8, 1.5) {$(\iota_{\ms \varSigma_{t}}^{*})^{-1}$}; 
		\node[above] at (5.75, 3) {$G$}; 
		\node[above] at (2.25, 3) {$G$}; 
		\node[below] at (4, 4.5) {$\varXi_{t}^{*}$};
	\end{tikzpicture}
	\captionof{figure}{The 
		time evolution map $U_{t}$ in terms of the causal propagator $G$ and the induced Killing flow $\varXi_{t}^{*}$. Here, $\cR$ is the Cauchy restriction map and $\iota_{\ms \varSigma}^{*}$ is the restriction map. 
	}
	\label{fig: time_evolution_causal_propagator_restriction_op_Killing_flow}
\end{center} 

In contemporary of this work, 
Capoferri and Murro~\cite[Thm. 1.1]{Capoferri} 
have obtained an oscillatory integral representation of $U_{t}$ (modulo smoothing operators) for the reduced massless Dirac equation (in the sense of Definition 3.5 of their paper) on a $4$-dimensional spatially compact globally hyperbolic spin-spacetime using the global phase function approach~\cite{Laptev_CPAM_1994} of Fourier integral operators. 
An antecedent of this analysis can be traced back to 
Capoferri and Vassiliev~\cite{Capoferri_2020} 
who have constructed $U_{t} (\tilde{D})$~\eqref{eq: def_U_t_Riem_Dirac} (modulo smoothing operators) as a summation of two invariantly defined oscillatory integrals,
global in space and in time, with distinguished complex-valued phase functions, when $\tilde{D}$ is a massless spin-Dirac operator on a $3$-dimensional closed Riemannian manifold. 

It is worthwhile to mention that one cannot deploy the algorithm used in~\cite{Capoferri, Capoferri_2020} in a straightforward way to derive the Fourier integral description of $U_{t}$ (Lemma~\ref{lem: symbol_Tr_U_0}) in the present setting. 
This is primarily because the Dirac Hamiltonian
\begin{equation} \label{eq: def_Dirac_Hamiltonian_SSST}
	H_{D} := - \frac{\ri}{\upbeta^{2} - \| \upalpha \|_{\fh}^{2} - \fh_{ij} \upalpha^{i} \fc (\rd t) \fc (\rd x^{j})}  
	\Big( \fc (\rd t) \big( \fh_{ij} \fc (\rd x^{i}) \nabla^{j} + U \big) + \fh_{ij} \upalpha^{i} \nabla^{j}  \Big)
\end{equation}
on a standard stationary spacetime is not of Dirac-type (albeit it is a first-order elliptic operator) which is one of the key assumptions of the hindmost literatures. 
%
%
%
%
%
%
%
%
%
%
\subsubsection{Principal symbol of $\Tr U_{t}$} 
If $\sU_{t} (x, y) := \fU_{t} (x, y) \, \rd t \otimes \sqrt{|\dVolh (x)|} \otimes \sqrt{|\dVolh (y)|}$ denotes the Schwartz kernel of $U_{t}$ then one considers the smoothed-out operator $\fU_{\rho}$  
\begin{equation} \label{eq: def_smoothed_time_evolution_op}
	\fU_{\rho} := \int_{\R} \fU_{t} \, \cF^{-1} (\rho) \, \rd t, 
	\quad 
	(\cF^{-1} \rho) (t) = \frac{1}{2 \pi} \int_{\spec L} \re^{\ri t \lambda} \rho (\lambda) \, \rd \lambda
\end{equation}
where $\rho$ is a Schwartz function on $\R$ such that $\supp (\cF^{-1} \rho)$ is compact. 
Let $\fU_{\rho} (x, x)$ represents the embedding of $\fU_{\rho} (x, y)$ to the diagonal in $\varSigma \times \varSigma$. 
The distributional trace $\Tr U_{t}$ is then obtained by 
\begin{equation} \label{eq: def_tr_smoothed_Cauchy_evolution_op}
	\Tr U_{\rho} := \int_{\varSigma} \tr \big( \fU_{\rho} (x, x) \big) \, \dVolh (x), 
\end{equation}
where $\tr : \End \sE \to \C$ is the endomorphism trace. 

In order to compute $\symb{\Tr U_{t}}$ we employ the bundle generalisation of 
Duistermaat and Guillemin's~\cite{Duistermaat_Inventmath_1975} 
idea, due to 
Sandoval~\cite{Sandoval_CPDE_1999}. 
One notices that the mapping $\sU_{t} (x, y) \mapsto \fU_{t} (x, x) \, \dVolh (x)$ can be viewed as the pullback 
\begin{equation}
	\varDelta^{*} : 
	C_{\rc}^{\infty} \big( \R \times \varSigma \times \varSigma; \varOmega_{\R} \otimes \halfDen_{\ms \varSigma \times \varSigma} \otimes \Hom{\sE_{\varSigma}, \sE_{\varSigma}} \big) 
	\to 
	C_{\rc}^{\infty} \big( \R \times \varSigma; \varOmega_{\R} \otimes \varOmega_{\ms \varSigma} \otimes \End \sE_{\varSigma} \big) 
\end{equation}
of $\sU_{t}$ via the diagonal embedding 
\begin{equation} \label{eq: def_diagonal_embedding_R_Cauchy}
	\varDelta : \R \times \varSigma \to \R \times \varSigma \times \varSigma. 
\end{equation}
It follows that $\varDelta^{*}$ is a Fourier integral operator of order $(d-1) / 4$ associated to the canonical relation 
$
C_{\varDelta^{*}} 
:= 
\{ (t, \tau; x, \xi + \eta; t, \tau; x, \xi; x, \eta) \in \dotCoTan \R \times \coTansM_{\varSigma} \times \dotCoTan \R \times \dotCoTansM_{\varSigma} \times \dotCoTansM_{\varSigma} \}
$~\cite[$(1.20)$]{Duistermaat_Inventmath_1975}. 
Hence, for a fixed $t$, $\tr (\varDelta^{*} \sU_{t})$ is density on $\varSigma$  which can be integrated. 
The integration over $\varSigma$ is the pushforward 
\begin{equation}
	\uppi_{*} : 
	C_{\rc}^{\infty} \big( \R \times \varSigma; \varOmega_{\R} \otimes \varOmega_{\ms \varSigma} \big) 
	\to 
	C^{\infty} \big( \R; \varOmega_{\R}), 
	~ \tr (\varDelta^{*} \sU_{t}) \mapsto \uppi_{*} \big( \tr (\varDelta^{*} \sU_{t}) \big) 
\end{equation}
of the Cartesian projection 
\begin{equation}
	\uppi : \R \times \varSigma \to \R,  
\end{equation}
which is also a Fourier integral operator of order $1/2 - (d-1)/4$ associated to the canonical relation 
$
C_{\uppi_{*}} 
:= 
\{ (t, \tau; t, \tau; x, 0) \in \dotCoTan \R \times \dotCoTan \R \times \coTansM_{\varSigma} \}
$~\cite[$(1.22)$]{Duistermaat_Inventmath_1975}. 
Therefore 
\begin{equation} \label{eq: Tr_U_t_pushforward_tr_pullback}
	\Tr U_{t} = \uppi_{*} \circ \tr (\varDelta^{*} \sU_{t}).   
\end{equation}
Alternatively, one can also integrate $\varDelta^{*} \sU_{t}$ over $\varSigma$ so that $\uppi_{*} (\varDelta^{*} \sU_{t})$ is a $\End \sE_{\varSigma}$-valued density on $\R$ by reckoning 
$
\uppi_{*} : 
C_{\rc}^{\infty} \big( \R \times \varSigma; \varOmega_{\R} \otimes \varOmega_{\varSigma} \otimes \End \sE_{\varSigma} \big) 
\to 
C^{\infty} \big( \R; \varOmega_{\R} \otimes \End \sE_{\varSigma})
$.   
Then taking the endomorphism trace one arrives at 
\begin{equation} \label{eq: Tr_U_t_tr_pushforward_pullback}
	\Tr U_{t} = \tr (\uppi_{*} \circ \varDelta^{*} \sU_{t}). 
\end{equation}
Hence, in this sense 
\begin{equation}
	\uppi_{*} \circ \tr \circ \varDelta^{*} = \tr \circ \uppi_{*} \circ \varDelta^{*}
\end{equation}
and one utilises the clean intersection (as detailed in Appendix~\ref{sec: density_canonical_relation}) between $\uppi_{*} \circ \varDelta^{*}$ and $\sU_{t}$ to compute the principal symbol of $\Tr U_{t}$.  
%
%
%
%
%
%
%
%
%
%
\subsubsection{Spectral theory} 
Albeit finite energy solutions of Dirac equation do not live in $L^{2}$ sections of $\sE$, $\ker D$ can be naturally given a hermitian structure by equipping with a Killing flow invariant hermitian form $\scalarProdTwo{\cdot}{\cdot}$ owing to Assumption~\ref{asp: trace_formula}~\ref{asp: hermitian_form_Dirac_op_SST}.  
It is noteworthy that the non-definiteness of $(\cdot|\cdot)$ is a characteristic feature of any Lorentzian spin-manifold: a positive-definite $(\cdot|\cdot)$ invariant under the spin group only exists for the Riemannian case. 
%
%
%
%
%
%
%
%
%
%
\subsubsection{Weyl law} 
There are a number of approaches to deriving the Weyl law 
(see e.g. the review~\cite{Ivrii_BullMathSci_2016}).  
Amongst those, we will use the 
Fourier-Tauberian argument~\cite{Hoermander_ActaMath_1968}  
(see also, e.g.~\cite[App. B]{Safarov_AMS_1997},~\cite{Safarov_JFA_2001} and references therein). 
The key idea is to relate the Weyl counting function with $\Tr U_{t}$ via the distributional Fourier transform  
\begin{equation} \label{eq: Weyl_counting_function_Tr_U_t}
	\frac{\rd \fN}{\rd \lambda} = \cF_{t \mapsto \lambda}^{-1} (\Tr U_{t})  
\end{equation}
and compute the right hand side using the express for $\Tr U_{0}$. 
%
%
%
%
%
%
%
%
%
%
\subsection{Organisation of the paper} 
We end this section with a literature review of pertinent results. 
Section~\ref{sec: setup} provides the geometric setup of this article where we have briefly collected the rudimentary backgrounds on globally hyperbolic stationary spacetimes and on  Dirac-type operators together with their causal propagators. 
Derivation of the trace formula for $U_{t}$ resp. the spectral theory of $L$ have been presented in Section~\ref{sec: trace_Cauchy_evolution_op} resp. in Section~\ref{sec: spectral_theory_L}. 
The announced results on singularity analysis of $\Tr U_{t}$ have been proven in Section~\ref{sec: proof_trace_formula_nondegenerate_fixed_pt} where we have computed $\symb{\Tr U_{T}}$ for trivial $T = 0$ and non-trivial $T \neq 0$ periodic orbits. 
The key technical tool of this computation is the composition of Fourier integral operators which has been reviewed in Appendix~\ref{sec: density_canonical_relation} for the convenience of readers. 
To maintain a smooth flow of the main text, we have placed the restriction map on a vector bundle in Appendix~\ref{sec: restriction_op}. 
%
%
%
%
%
%
%
%
%
%
\subsection{Literature} 
\label{sec: literature}
The field of asymptotic (semiclassical) trace formulae stems from the study of Green's operator for a Schr\"{o}dinger operator in the limit of vanishing Planck's constant by  
Gutzwiller ~\cite{Gutzwiller_JMP_1971}.   
This seminal work is not entirely rigorous; see for instance, the 
expositions~\cite{Uribe_Cuernavaca_1998, Muratore-Ginanneschi_PR_2003} 
for a scrutinised discussion of his original idea.   
A number of mathematically diligent 
proofs 
(see e.g.~\cite[Thm. 3]{Meinrenken_ReptMathPhys_1992} 
and the expository articles~\cite{Uribe_Cuernavaca_1998, Verdiere_AIF_2007} 
with references therein)  
have been reported since then.  
Amongst these, 
Chazarain~\cite{Chazarain_InventMath_1974, Chazarain_CPDE_1980} (for Laplace-Beltrami operator on a closed Riemannian manifold)
and 
Duistermaat-Guillemin~\cite{Duistermaat_Inventmath_1975} (for a scalar, elliptic, selfadjoint and positive first-order pseudodifferential operator $P$ on a compact manifold without boundary)
deployed global Fourier integral operators to derive the complete singularity structure of the wave-trace. 
We refer to the  
monographs~\cite[Chap. XXIX]{Hoermander_Springer_2009},~\cite[Chap. 1]{Safarov_AMS_1997},~\cite[Chap. 11]{Guillemin_InternationalP_2013} 
for details.  

Bolte and Keppeler~\cite{Bolte_PRL_1998, Bolte_AnnPhys_1999} 
have generalised Gutzwiller's work for spin Dirac operators on Minkowski spacetime 
(see also the elucidating articles~\cite{Bolte_FoundPhys_2001, Bolte_JPA_2004} 
and references therein for chronological developments). 
In contemporary to their reports, mathematically rigorous analysis has been reported by 
Sandoval~\cite{Sandoval_CPDE_1999} 
promoting the Duistermaat-Guillemin framework for Dirac-type operators on a closed Riemannian manifold. 
Her result on Dirac-wave-trace invariants at the trivial period is closely related to the evaluation of the residues of the eta-invariant by 
Branson and Gilkey~\cite{Branson_JFA_1992}  
and the behaviour of eigenfunctions in the high energy limit by 
Jakobson and Strohmaier~\cite{Jakobson_CMP_2007}.  
In the particular case of a massless Dirac operator on a closed $3$-dimensional Riemannian spin  manifold,  
Capoferri and Vassiliev~\cite{Capoferri_2020}
have computed the third local Weyl coefficients by advancing the framework by  
Chervova \textit{et al}.~\cite{Chervova_JST_2013} 
(see also the review by 
Avetisyan \textit{et al}.~\cite{Avetisyan_JST_2016})
for asymptotic spectral analysis of a general elliptic first-order system;   
see~\cite[Sec. 11]{Chervova_JST_2013} 
for a bibliographic overview.
An in-depth investigation of different spectral coefficients of a selfadjoint Laplace-type operator on a hermitian vector bundle over a closed Riemannian manifold has been performed by 
Li and Strohmaier~\cite{Li_JGP_2016}.   
In particular, they have obtained the relevant coefficients for Dirac-type operators with a more general bundle endomorphism than in  
Sandoval~\cite{Sandoval_CPDE_1999} (identity endomorphism)
and in 
Branson-Gilkey~\cite{Branson_JFA_1992} (scalar endomorphism),  
and bridged the results by 
Chervova \textit{et al}.~\cite{Chervova_JST_2013} 
(and their follow up works as mentioned in~\cite{Li_JGP_2016}) 
with the known heat-trace invariants.  

The primary and common ingredient of~\cite{Bolte_PRL_1998, Bolte_AnnPhys_1999, Sandoval_CPDE_1999, Capoferri_2020} 
is to determine the time evolution operator (modulo smoothing operators) by solving the transport equations order by order.   
Bolte-Keppeler have also identified terms responsible for spin-magnetic and spin-orbit interactions in the semiclassical expression. 
However, they have finally considered a regularised truncated time evolution operator by introducing an energy localisation to deal with the continuous spectrum of Dirac Hamiltonian arising due to the non-compact Minkowski spacetime. 
Such restrictions were absent in the rest of the hindmost references as they considered closed Riemannian manifolds and utilised the full power of Fourier integral operator theory in contrast to Bolte-Keppeler who have worked with oscillatory integrals. 
A novel feature in the lines of research by 
Vassiliev and his collaborators~\cite{Chervova_JST_2013, Avetisyan_JST_2016} 
is the second term of the Weyl law (see the references cited in these papers for earlier works). 
Analogous results have been achieved by  
Li-Strohmaier~\cite{Li_JGP_2016} 
employing a different spectral analysis.  
Specifically, the global phase function approach~\cite{Laptev_CPAM_1994} 
of Fourier integral operators has been deployed in~\cite{Capoferri_2020} 
to construct (modulo smoothing operators) the solution operator of a massless spin-Dirac operator on a $3$-dimensional closed Riemannian manifold. 
They have also provided a closed formula for the principal symbol and an algorithm for computing the subprincipal symbol of this operator.

A general relativistic generalisation of Duistermaat-Guillemin-Gutzwiller trace formula has been initiated by 
Strohmaier and Zelditch~\cite{Strohmaier_AdvMath_2021, Strohmaier_IndagMath_2021}  
(see also the review~\cite{Strohmaier_RMP_2021}) 
who have studied the d'Alembertian on a globally hyperbolic spatially compact stationary spacetime. 
Their crucial step was to set up a relativistic description of the classical and the quantum dynamics and advance the celebrated 
Duistermaat-Guillemin~\cite{Duistermaat_Inventmath_1975} 
framework accordingly. 
In particular, they have expressed the time evolution operator by means of the causal propagator of d'Alembertian and Killing flow, and computed its principal symbol by utilising the symbolic calculus of Fourier integral operators based on 
Duistermaat and H\"{o}rmander's~\cite{Duistermaat_ActaMath_1972} 
classic work of distinguished parametrices.
Subsequently, they employed 
Guillemin's symplectic residue~\cite{Guillemin_AdvMath_1985} 
approach at trivial period and tailored the Duistermaat-Guillemin computation for the non-trivial periods. 
Consequently, the Weyl law in the space of lightlike geodesic strips has been reported by them using the standard Fourier-Tauberian argument.      
Apart from the current report, their work has been generalised in a bundle setting for a d'Alembertian on a globally hyperbolic stationary Kaluza-Klein spacetime by 
McCormick~\cite{McCormick} 
who utilised some technical results of this article. 
%
%
%
%
%
%
%
%
%
%
\section*{Acknowledgement}  
The author is indebted to Alexander Strohmaier for suggesting and supervising this problem as a part of his PhD project at the University of Leeds where the work was carried out and supported by the Leeds International Doctoral Studentship.  
It is a pleasure to thank Benjamin Sharp for commenting on the presentation of the manuscript and Yan-Long Fang and Christian B\"{a}r for fruitful discussions. 
I would like to acknowledge the anonymous referees for their scrutinisation and invaluable suggestions. 
%
%
%
%
%
%
%
%
%
%
\section{The setup} 
\label{sec: setup}
\subsection{Globally hyperbolic stationary spacetime}
\label{sec: GHSST}
Let $(M, \fg)$ be a $d \geq 2$ dimensional globally hyperbolic spacetime.
This means that this spacetime is isometrically diffeomorphic to the product manifold
\begin{equation}
	\spacetime \cong (\R \times \varSigma_{t}, \upbeta \, \rd t^{2} - \fh_{t}), 
\end{equation}
where $\upbeta \in C^{\infty} (\R \times \varSigma, \R_{>0})$ is the lapse function, $t \in C^{\infty} (\R \times \varSigma, \R)$ is the natural projection, each level set  $\varSigma_{t} := {x \in \sM | \boldsymbol{t} (x) = t}$ of a global Cauchy temporal function $\boldsymbol{t}$ is a Cauchy hypersurface, and $(\fh_{t})_{t}$ is a $1$-parameter family of Riemannian matrices on $\varSigma_{t}$ that varies smoothly with $t$. 

A stationary spacetime $\SST$, per se, is an oriented and time-oriented Lorentzian manifold $(\sM, \fg)$ admitting a global timelike Killing flow $\varXi : \R \times \sM \to \sM$. 
This flow is physically interpreted as the flow of time and it allows a canonical yet non-unique global (resp. local) splitting of the spacetime manifold $\sM$ (resp. metric $\fg$).  
Not all stationary spacetimes are globally hyperbolic and we restrict to only those  which are. 
The 
survey article~\cite{Sanchez_AMS_2011}  
and the original references cited therein is referred for different equivalent characterisations of global hyperbolicity and the necessary and sufficient conditions for a stationary spacetime to be globally hyperbolic.  

Note, on a generic $d > 2$-dimensional (standard) stationary spacetime the Killing vector field $(\partial_{t} =) Z$ is not orthogonal to $\varSigma_{t}$ because the corresponding $1$-form $Z^{\flat}$ does not satisfy the hypersurface-orthogonality condition: $Z^{\flat} \wedge \rd Z^{\flat} \neq 0$ (i.e., the orthogonal geometric distribution of $Z$ is non-involutive). 
Physically this means that the neighbouring orbits of $Z$ can twist around each other. 
In $d=2$ this cannot happen, i.e., every Killing vector field is at least locally hypersurface-orthogonal. 
If one imposes the condition that $\varSigma_{t}$ is orthogonal to the orbits of the spacetime isometry, then $\SST$ is called a static spacetime and one has a canonical non-unique global time-coordinate $t$. 
In this case the shift vector field $\upalpha$ vanishes identically so that there is no $\rd t \, \rd x^{i}$-type cross terms in~\eqref{eq: def_spacetime_metric_SSST} and the Riemannian metric $\fh$ becomes independent of time. 
Additionally, if we demand that $Z$ has a constant norm, then a static spacetime is called an ultrastatic spacetime. 
This enforces the lapse function $\upbeta$ to be the identity function. 
%
%
%
%
%
%
%
%
%
%
\subsection{Dirac-type operators}
\label{sec: Dirac_op}
Suppose that $\big( \sE \to \sM, (\cdot|\cdot) \big)$ is a vector bundle over a spacetime $(\sM, \fg)$, endowed with a sesquilinear form $(\cdot|\cdot)$ and that $D$ is a Dirac-type operator on $\sE$, symmetric with respect to $(\cdot|\cdot)$.  
Therefore, $D^{2}$ is a normally hyperbolic operator and by the polarisation identity 
(see e.g.~\cite[Rem. 2.19]{Baer_Springer_2012})  
\begin{equation} \label{eq: anticommutation_symbol_Dirac_type}
	[\symb{D} (x, \xi), \symb{D} (x, \eta)]_{+} = 2 \, \fg_{x}^{-1} (\xi, \eta) \, \one_{\End{\sE}_{x}}, \quad \forall \xxi, (x, \eta) \in C^{\infty} (\sM; \coTansM). 
\end{equation}
Moreover, $\symb{D}$ defines a Clifford action of $\coTansM$ on $\sEsM$ by 
\begin{equation} \label{eq: def_symbol_Dirac_type_op}
	\symb{D} (\rd f) := \ri \, [D, f]_{-}, 
	\quad 
	\forall f \in C^{\infty} (\sM), 
\end{equation}
which turns $\big( \sE \to \sM, (\cdot|\cdot) \big)$ into a bundle of Clifford modules $\big( \sE \to \sM, (\cdot|\cdot), \symb{D} \big)$ over $\spacetime$. 
Furthermore,~\eqref{eq: def_symbol_Dirac_type_op} gives the Leibniz rule for $D$ and the Clifford mapping~\eqref{eq: anticommutation_symbol_Dirac_type} defines the (pointwise) Clifford multiplication  
\begin{equation} \label{eq: def_Clifford_multiplication}
	\fc \big( \xxi \otimes u \big) := \symb{D} \xxi \, (u), \quad \forall \xxi \in C^{\infty} (\sM; \coTanM), \forall u \in \secsME. 
\end{equation}
By the Weitzenb\"{o}ck formula 
(see e.g.~\cite[Prop. 3.1]{Baum_AGAG_1996}),   
given a Dirac-type operator $D$ on any vector bundle $\big( \sE \to \sM, (\cdot|\cdot) \big)$, there exists a unique $(\cdot|\cdot)$-compatible connection $\nabla$ on $\sE$, called the Weitzenb\"{o}ck connection, and a unique potential $V \in C^{\infty} \big( \sM; \End{\sE} \big)$ such that 
\begin{equation} \label{eq: Weitzenboeck_formula_Dirac_type}
	D^{2} = -\tr_{\fg} \big( (\nabla^{\mathrm{LC}} \otimes \one_{\sE} + \one_{\coTanM} \otimes \nabla) \circ \nabla \big) + V,    
\end{equation}
where $\tr_{\fg} : C^{\infty} (\sM; \coTansM \otimes \coTansM) \to C^{\infty} (\sM)$ denotes the metric trace: $\tr_{\fg} \big( \xxi \otimes (x, \eta) \big) := \fg_{x}^{-1} (\xi, \eta)$ and $\nabla^{\mathrm{LC}}$ is the Levi-Civita connection on $\sM$. 
The most general Dirac-type operator on a Clifford module $\big( \sE \to \sM, (\cdot|\cdot), \symb{D} \big)$ then has the form 
\begin{equation} \label{eq: Dirac_op}
	D = - \ri \, \fc \circ \nabla + U,    
\end{equation}
where the potential term $U \in C^{\infty} \big( \sM; \End{\sE} \big)$ is defined by the hindmost equation. 
The preceding assumption entails that $U$ and $- \ri \nabla_{X}$ are symmetric with respect to $(\cdot|\cdot)$ provided that $X$ is divergence free. 

The Weitzenb\"{o}ck connection $\nabla$ induces a connection $\nabla^{\Hom{\sE, \sE}}$ on the homomorphism bundle $\Hom{\sE, \sE} \to \sM$ over the spacetime, which is not, in general, a Clifford connection as can be seen readily by taking covariant derivative of~\eqref{eq: anticommutation_symbol_Dirac_type} with respect to this connection and properties of the Levi-Civita covariant derivative. 
In particular, locally one can always choose an orthonormal frame $\{ e_{i} \}$ and its coframe $\{ \varepsilon^{i} \}$ on $\sM$, and express $D = - \ri \, \slashed{\nabla} + U$, where we have use the Feynman-slash notation: $\slashed{\nabla} := \symb{D} (\varepsilon^{i}) \nabla_{e_{i}}$, i.e., $\nabla$ is composed with Clifford multiplication and then traced over. 
A bit lengthy yet straightforward computation yields 
\begin{eqnarray}
	D^{2} 
	& = &  
	-\tr_{\fg} \big( (\nabla^{\ms \mathrm{LC}} \otimes \one_{\sE} + \one_{\coTanM} \otimes \nabla) \circ \nabla \big) - \frac{1}{2} \slashed{\sR} + U^{2} - \ri \slashed{\nabla} (U)  
	\nonumber \\ 
	&& 
	- \Big( \slashed{\nabla}^{\ms \Hom{\sE, \sE}} \big( \symb{D} (\varepsilon^{j}) \big) + \slashed{\Gamma}^{j} + \ri U \symb{D} (\varepsilon^{j}) + \ri \symb{D} (\varepsilon^{j}) U \Big) \nabla_{e_{j}},  
\end{eqnarray}
where $\sR$, $\Gamma$ and $- \slashed{\sR} / 2 := - \symb{D} (X^{\flat}) \, \symb{D} (Y^{\flat}) \, \sR_{X, Y} / 2$ are the curvature, connection $1$-form and  Weitzenb\"{o}ck curvature of $\nabla$, respectively.  
Since $\nabla$ is a Weitzenb\"{o}ck connection, comparing with~\eqref{eq: Weitzenboeck_formula_Dirac_type} we equate the coefficients of $\nabla_{e_{j}}$ in the foregoing equation of $D^{2}$ to zero to obtain the covariant derivative of $\symb{D}$:   
\begin{equation} \label{eq: covariant_derivative_symbol_Dirac} 
	\slashed{\nabla}^{\ms \Hom{\sE, \sE}} \big( \symb{D} (\varepsilon^{j}) \big) = - \slashed{\Gamma} - \big[ U, \symb{D} (\varepsilon^{j}) \big]_{+}, 
	\quad 
	\forall j = 1, \ldots, d. 
\end{equation}

%
%
%
\begin{remark} \label{rem: Clifford_connection} 
	A Clifford connection $\tilde{\nabla}^{\ms \Hom{\sE, \sE}}$, characterised by 
	\begin{equation}
		\big[ \tilde{\nabla}_{X}^{\ms \Hom{\sE, \sE}}, \symb{D} (x, \eta) \big]_{-} = \symb{D} \big(x, \nabla_{X}^{\ms \mathrm{LC}} \eta \big), 
		\quad 
		\forall (x, X) \in C^{\infty} (\sM; \tangent \sM), \forall (x, \eta) \in C^{\infty} (\sM; \coTansM),  
	\end{equation}
	always exists and then the Clifford module bundle $\big( \sE \to \sM, (\cdot|\cdot), \symb{D}, \tilde{\nabla} \big)$ is called the Dirac module bundle where the most general Dirac-type operator has the form 
	\begin{equation}
		\tilde{D} := - \ri \, \fc \circ \tilde{\nabla} + \tilde{U}
	\end{equation}
	for all $\tilde{U} \in C^{\infty} (\sM; \End{\sE})$ such that (cf.~\eqref{eq: covariant_derivative_symbol_Dirac}) 
	\begin{equation} \label{eq: potential_compatible_Dirac}
		[\tilde{U}, \symb{D} (\cdot)]_{-} = 0. 
	\end{equation}
	The operator $\tilde{D}$ has the same principal symbol as $D$ so that they differ only by a smooth term. 
	Since $\tilde{\nabla}$ is compatible with the Clifford multiplication~\eqref{eq: def_Clifford_multiplication}, $\tilde{D}$ is sometimes called compatible Dirac-type operator. 

	Note, the spin connections induced by the Levi-Civita connection are example of a Clifford connection and the corresponding massive 
	spin-Dirac (see e.g.~\cite[Exm. 2.21]{Baer_Springer_2012}) 
	and 
	twisted spin-Dirac operators (see e.g.~\cite[Exm. 2.22]{Baer_Springer_2012}) 
	are indeed compatible Dirac-type operators where $\tilde{U}$ models the mass (in the appropriate limit) term. 
\end{remark}
%
%
%

The condition~\eqref{eq: covariant_derivative_symbol_Dirac} and the Remark~\ref{rem: Clifford_connection} are originally due to 
Branson and Gilkey~\cite[Lem. 2.3]{Branson_JFA_1992} 
who have worked on Riemannian setting and considered $\tilde{U} \equiv 0$. 
%
%
%
%
%
%
%
%
%
%
\subsection{Lie derivative on $\secsME$}
\label{sec: Lie_derivative}
On a stationary spacetime $\SST$, the cotangent lift $\coTan \varXi_{s}$ of the spacetime isometry $\varXi_{s}$ naturally induces a $\R$ action on $\coTansM$: $\coTan \varXi (s; y, \eta) := (\coTan_{x} \varXi_{s}) \eta$. 

Furthermore, if $\hat{\varXi}$ denotes the fibrewise canonical isomorphism $(\varXi^{*} \sE)_{(s, x)} \cong \sE_{\varXi_{s} (x)}$ then the pullback $\varXi^{*} : \comSecE \to C^{\infty} (\R \times \sM; \varXi^{*} \sE)$ via the morphism $(\varXi, \hat{\varXi})$ is a Fourier integral operator whose Schwartz kernel is given by 
(see e.g.~\cite[(2.4.4), (2.4.22)]{Duistermaat_Birkhaeuser_2011})    
\begin{subequations} \label{eq: pullback_kernel}
	\begin{eqnarray}
		&& 
		\mathsf{\Xi} \in I^{-1/4} \big( \R \times \sM \times \sM, \varGamma'; \Hom{\sE, \varXi^{*} \sE} \big), 
		\label{eq: pullback_kernel_Lagrangian_dist}
		\\ 
		&& 
		\varGamma := \big\{ \big( s, \tau; x, \xi; y, - \eta \big) \in \dotCoTan \R \times \dotCoTansM \times \dotCoTansM \,|\, \tau = - \xi (Z), \xxi = (\coTan_{x} \varXi_{s}) \yeta \big\}, \qquad \quad 
		\label{eq: def_canonical_relation_pullback}
		\\ 
		&& 
		\symb{\mathsf{\Xi}} := (2 \pi)^{1/4} \one_{\Hom{\sE, \varXi^{*} \sE}} \sqrt{|\dVol_{\ms \varGamma}|} \otimes \bbl,  
		\label{eq: symbol_pullback}
	\end{eqnarray}
\end{subequations}
where $\dVol_{\ms \varGamma}$ is the volume form on the homogeneous canonical relation $\varGamma \subset \dotCoTan (\R \times \sM) \times \dotCoTansM$ and $\bbl$ is a section of the Keller-Maslov bundle $\bbL_{\varGamma} \to \varGamma$ over $\varGamma$, constructed as below. 
One observes that $\varGamma$ is the graph of $\coTan \varXi_{s}$ and at $s=0$, $\varGamma$ is essentially the conormal bundle $\varGamma_{0} := \{ (0, \tau) \} \times (\varDelta \, \dotCoTanM)'$. 
Since $\dotCoTansM$ is a symplectic manifold, it admits the global volume induced by the canonical symplectic form on $\dotCoTansM$. 
Then $\dVol_{\ms \varGamma_{0}}$ is obtained via the pullback of the projector $\Pr : \varGamma_{0} \to \R \times \dotCoTansM$, which is invariant under the flow $\varPsi_{s}$ of the Hamiltonian vector field generated by the extended Hamiltonian $\tau + \xi (Z)$ and given by $\dVol_{\ms \varGamma} = \rd s \otimes \rd x \wedge \rd \xi$ in the parametrisation~\eqref{eq: def_canonical_relation_pullback}. 
To construct $\bbL_{\varGamma}$ we re-use the fact that $\varPsi_{s}$ sweeps out $\{ (0, \tau) \} \times (\varDelta \, \dotCoTanM)'$ to $\varGamma$ and hence $\bbL_{\varGamma}$ is constructed by parallelly transporting the sections of $\bbL_{0}$ along the orbits of $\varPsi_{s}$ where $\bbL_{0}$ the Keller-Maslov bundle over $\varGamma_{0}$ consists of global constant section   
(see e.g.~\cite[Sec. 4]{Meinrenken_ReptMathPhys_1992},~\cite[Sec. 5.13]{Guillemin_InternationalP_2013}). 

More generally, $\varXi^{*}$ can be extended to a sequentially continuous linear map on $I^{m} (\sM, \mathcal{L}; \sE)$ 
by
\begin{equation} \label{eq: pullback_Killing_flow_Lagrangian_dist}
	\varXi^{*} : I^{m} (\sM, \mathcal{L}; \sE) \to I^{m - 1/4} (\R \times \sM, \varGamma' \circ \mathcal{L}; \varXi^{*} \sE), 
	\qquad 
	\symb{\varXi^{*} u} \equiv \symb{\mathsf{\Xi}} \diamond \symb{u} 
\end{equation}
for any $u \in I^{m} (\sM, \mathcal{L}; \sE)$. 
The composition $\diamond$ of principal symbols is presented in details in 
Appendix~\ref{sec: density_canonical_relation} and the equation of $\symb{\varXi^{*} u}$ is in the sense of modulo Keller-Maslov part.  

The induced Killing flow $\varXi_{s}^{*}$ paves the way to define the Lie derivative $\pounds_{\! Z}$ on $E$ with respect to the Killing vector field $Z$: 
\begin{equation} \label{eq: def_Lie_deri}
	\pounds_{\! Z} u := \frac{\rd}{\rd s} \Big|_{s=0} \varXi_{s}^{*} u.  
\end{equation}
Note, $\pounds_{\! Z}$ is essentially a generalisation of the 
Lichnerowicz spinor Lie derivative 
for stationary spacetimes when one does not necessarily have a spin-structure. 
%
%
%
%
%
%
%
%
%
%
\subsection{Classical dynamics}
\label{sec: classical_dyanmics}
The primary tenet of the semiclassical analysis is to connect the relativistic trace formula with its classical dynamics. 
In non-relativistic mechanics, the cotangent bundle $\coTanCauchy$ models the classical phase space, whereas the Hilbert-space quantum dynamics takes place in $L^{2} (\varSigma)$. 
The naive expectation of using this pair or the pair $\big( \coTansM, L^{2} (\sM; \sE) \big)$ does not work because the former depends on the choice of Cauchy hypersurface $\varSigma \subset \sM$ and for the latter pair, the (Killing flow invariant) sesquilinear form $(\cdot|\cdot)$ on $\sE$ (Assumption~\ref{asp: trace_formula}~\ref{asp: Direc_op_symmetric_SST}) does not induce any $L^{2}$-norm. 
One can, of course, choose an arbitrary hermitian form in order to have a $L^{2}$ space on $\sE$, but then this $L^{2}$-space will depend on the particular choice of the hermitian form as $\sM$ is non-compact. 
In pursuance of defining the correct classical dynamics, one notes that $\Char D$ is the lightcone bundle $\coLightBun$. 
Since $\Char{D}$ and $\Char{\square}$ are identical, the classical dynamics in this case coincides with  that in the 
Strohmaier-Zelditch trace formula~\cite{Strohmaier_AdvMath_2021} 
and hence we adopt their formulae.   
The metric-Hamiltonian $H_{\fg}$ is a homogeneous function of degree $2$ in the cotangent fibres. 
Referring to this as the dilation, let $\E$ be the Euler vector field which is the generator of this action. 
On $\coLightBun$, clearly $H_{\fg}$ vanishes and we have $[\E, X_{\fg/2}]_{-} = X_{\fg/2}$, where $X_{\fg/2}$ is the Hamiltonian vector field of $H_{\fg}$. 
Then the Hamiltonian reduction of $\coLightBun$ is the space of scaled-lightlike geodesic strips $\cN$. 
That is, if $\spacetime$ is geodesically complete then $\cN$ is the quotient of $\coLightBun$ by the $\R$-group action generated by $X_{\fg/2}$. 
Similarly, by taking the quotient of $\cN$ by the $\R_{+}$-group action generated by $\E$ we obtain the space of unparametrised-lightlike geodesic strips $\tilde{\cN}$. 
If $\spacetime$ is a spatially compact globally hyperbolic spacetime then $\tilde{\cN}$ is a \textit{conic compact contact manifold} whose symplectisation is the conic symplectic manifold $\cN$ induced from the conic contact manifold $\coLightBun$~\cite[pp. 10-12]{Penrose_1972},~\cite[Thm. 2.1]{Khesin_AdvMath_2009}. 

We remark that $\cN$ is defined \textit{invariantly} and~\cite[Prop. 2.1]{Strohmaier_AdvMath_2021}
\begin{equation} \label{eq: symplectic_diffeo_lightlike_geodesic_coTanCauchy}
	\varsigma: \cN \to \dotCoTanCauchy 
\end{equation}
is a homogeneous symplectomorphism. 
Furthermore, the geodesic flow on any spatially compact globally hyperbolic spacetime $\spacetime$ does not necessarily have to be complete. 

Recall that on a standard stationary spacetime, a geodesic $c : \R \to \R \times \varSigma$ can be expressed as $c (s) = \big( t (s), c_{\ms \varSigma} (s) \big)$ and then a periodic one with period $T$ means  
\begin{equation}
	c_{\ms \varSigma} (s + T) = c_{\ms \varSigma} (s), 
	\quad 
	\forall s \in \R.  
\end{equation}
The cotangent lift $\gamma$ of such periodic curves $c$ on $\cN$ is vital for the Duistermaat-Guillemin-Gutzwiller trace formula. 
As evident, these are periodic lightlike geodesic strips $\gamma$ of length $T$. 
We note that the set of $T$-periodic curves $\gamma$ on $\cN$ is precisely the length spectrum of $(\varSigma, \fh)$ for an ultrastatic spacetime. 
For details, see, for 
instance~\cite{Sanchez_AMS_1999, Bartolo_NonlinearAnal_2001} 
and the earlier references cited therein. 
%
%
%
%
%
%
%
%
%
%
\subsection{Green's operators}
Since $D^{2}$ is a normally hyperbolic operator, it admits unique advanced $F^{\adv}$ and retarded $F^{\ret}$ Green's operators on a globally hyperbolic spacetime $\spacetime$ 
(see e.g.~\cite[Cor. 3.4.3]{Baer_EMS_2007}). 
Therefore, there exist unique advanced $\advGreenOp := D F^{\adv}$ and retarded $\retGreenOp := D F^{\ret}$ Green's operators for $D$ on $\spacetime$~\cite[Thm. 1]{Muehlhoff_JMP_2011}.   
Their antisymmetric combination 
\begin{equation} \label{eq: def_causal_propagator_Dirac}
	G := \retGreenOp - \advGreenOp : \comSecsME \to C_{\mathrm{sc}}^{\infty} (\sM; \sE)
\end{equation}
defines the Pauli-Jordan-Lichnerowicz operator, also known as the \textit{causal propagator}. 

We have assumed that $D$ is symmetric with respect to the sesquilinear form $(\cdot|\cdot)$ on $\sE$. 
Hence, it follows that 
\begin{equation} \label{eq: adv_ret_Green_op_Dirac_type_selfadjoint_sesquilinear}
	(G^{\adv, \ret} u | v) = (u | G^{\ret, \adv} v), 
	\quad 
	(G u | v) = - (u | G v), 
	\quad 
	\forall u, v \in \comSecsME,   
\end{equation}
due to the fact that $(F^{\adv, \ret} u | v) = (u | F^{\ret, \adv} v)$ 
(see e.g.~\cite[Lem. 3.4.4]{Baer_EMS_2007}). 

It is well-known that the causal propagator $F := F^{\ret} - F^{\adv}$ for $D^{2}$ is a Fourier integral operator, i.e., its Schwartz kernel $\fF$ is a Lagrangian distribution  
(see e.g.~\cite[Rem. 2.15]{Islam}): 
\begin{subequations}
	\begin{eqnarray}
		&& 
		\fF \in I^{-3/2} \big( \sM \times \sM, C'; \Hom{\sE, \sE} \big), 
		\\ 
		&& 
		\symb{\fF} = \frac{\ri}{2}\sqrt{2 \pi} w \sqrt{|\dVol_{\ms C}|} \otimes \mathbbm{l}, 
		\label{eq: symbol_causal_propagator_NHOp}
		\\ 
		&& 
		\Char{\fF} \cap C = \emptyset.   
	\end{eqnarray}
\end{subequations} 
Here $\dVol_{\ms C}$ is the natural density on the geodesic relation 
\begin{equation}
	C := \big\{ \xxiyeta \in \coLightBun \times \coLightBun \,|\, \exists ! s \in \R : \xxi = \varPhi_{s} \yeta \big\},     
	\label{eq: def_geodesic_relation} 
\end{equation}
where $\varPhi$ is the geodesic flow on the cotangent bundle $\coTanM$ restricted to the lightcone bundle $\coLightBun$. 
The construction of $\dVol_{\ms C}$ is originally due to 
Duistermaat-H\"{o}rmander~\cite[p. 230]{Duistermaat_ActaMath_1972} 
for a generic manifold which simplifies considerably for a globally hyperbolic spacetime $\spacetime$ as reported by 
Strohmaier-Zelditch~\cite[$(52)$, Rem. 7.1]{Strohmaier_AdvMath_2021}. 
By definition, for each $\xxiyeta \in C$ there is a unique $s \in \R$ such that $\xxi = \varPhi_{s} \yeta$ so that $C$ can be identified with an open subset of $\R \times \coLightBun$, where $s \in \R$ is the flow parameter. 
Recall that the Hamiltonian $H_{\fg}$-reduction of $\coLightBun$ is the conic symplectic manifold $\cN$ of scaled-lightlike geodesic strips~\cite[pp. 10-12]{Penrose_1972},~\cite[Thm. 2.1]{Khesin_AdvMath_2009}.  
Denoting by $\tilde{s}$, the dilation parameter on $\coLightBun$, the natural half-density on $C$ is given by 
\begin{equation}
	\sqrt{|\dVol_{\ms C}|} := \sqrt{|\rd s|} \otimes \sqrt{|\rd \tilde{s}|} \otimes \sqrt{| \dVol_{\ms \cN}|}. 
\end{equation}  
Note, \textit{this density differs from that by Duistermaat-H\"{o}rmander by a factor of $2$ because they used the Hamiltonian flow of $\fg^{-1}$ to parametrise $\cN$, in contrast to the flow of the Hamiltonian vector field $X_{\fg/2}$ generated by} $H_{\fg}$. 
Moreover, 
\begin{equation}
	\pounds_{X_{\fg}} \dVol_{\ms C} = 0. 
\end{equation}
The densities on forward (backward) geodesic relations 
\begin{equation}
	C^{\pm} := \big\{ \xxiyeta \in \coLightBun \times \coLightBun \,|\, \exists ! s \in \R_{\gtrless} : \xxi = \varPhi_{s} \yeta \big\} 
\end{equation}
follow from the fact that $\coLightBun = \dotCoTan_{0, +} \sM \sqcup \dotCoTan_{0, -} \sM$ in $d \geq 3$. 
When $d = 2$ then $\coLightBun$ has $4$ connected components and those can be incorporated similarly. 
In~\eqref{eq: symbol_causal_propagator_NHOp}, $\bbl$ is a section of the Keller-Maslov bundle $\bbL_{C} \to C$, and $w$ is the unique element of $C^{\infty} \big( C; (\pi^{*} \Hom{\sE, \sE}) |_{C} \big)$ that is diagonally the identity endomorphism and off-diagonally covariantly constant  
\begin{equation} \label{eq: symbol_causal_propagator_NHOp_covarinatly_constant}
	\connectionEndPiE_{X_{\fg/2}} w = 0   
\end{equation}
with respect to the $D^{2}$-compatible Weitzenb\"{o}ck covariant derivative $\connectionEndPiE_{X_{\fg/2}}$ (Definition~\ref{def: P_compatible_connection}) along the geodesic vector field $X_{\fg/2}$. 
%
%
%
\begin{remark} \label{rem: HVF_canonical_relation}
	By Definition~\ref{def: P_compatible_connection}, $X_{\fg}$ acts on $C^{\infty} \big( \coTansM; \pi^{*} \Hom{\sE, \sE} \big)$. 
	Thus~\eqref{eq: symbol_causal_propagator_NHOp_covarinatly_constant} 
	must be interpreted in the sense of the induced Hamiltonian vector field on $C$ by the vector field $(X_{\fg}, 0)$ on $(\coTansM \times \coTansM)$~\cite[Rem. 1, p. 216]{Duistermaat_ActaMath_1972} 
	(see also~\cite[Rem. 1, p. 69]{Hoermander_Springer_2009}). 
\end{remark}
%
%
%
By construction, $\fG = D \fF$. 
Then the preceding information entail that $G$ is a Fourier integral operator associated with the canonical relation $C$.  
%
%
%
\begin{lemma} \label{lem: causal_propagator_Dirac}
	Let $\big( \sE \to \sM, \symb{D} \big)$ resp. $\pi : \dotCoTansM \to \sM$ be a bundle of Clifford modules resp. the punctured cotangent bundle over a globally hyperbolic spacetime $\spacetime$ and $D$ a Dirac-type operator on $\sE$ whose principal symbol is denoted by $\symb{D}$. 
	The Schwartz kernel $\fG$ of the causal propagator of $D$ is then  
	\begin{subequations} \label{eq: causal_propagator_Dirac}
		\begin{eqnarray}
			&& 
			\fG \in I^{-1/2} \big( \sM \times \sM, C'; \Hom{\sE, \sE} \big), 
			\label{eq: causal_propagator_Dirac_FIO}
			\\  
			&& 
			\symb{\fG} = \frac{\ri}{2} \sqrt{2 \pi} \, \symb{D} \circ w \, \sqrt{| \dVol_{\ms C} |} \otimes \mathbbm{l},   
			\label{eq: symbol_causal_propagator_Dirac}
		\end{eqnarray}
	\end{subequations}
	where $\dVol_{\ms C}$ is the natural volume form on the geodesic relation $C$, $\bbl$ is a section of the Keller-Maslov bundle $\bbL_{C} \to C$ 
	(as constructed in~\cite[pp. 231-232]{Duistermaat_ActaMath_1972}), and $w$ is the unique element of $C^{\infty} \big( C; (\pi^{*} \Hom{\sE, \sE}) |_{C} \big)$ that is diagonally the identity endomorphism and off-diagonally covariantly constant  
	\begin{equation} \label{eq: symbol_causal_propagator_Dirac_covarinatly_constant}
		\connectionEndPiE_{X_{\fg/2}} w = 0   
	\end{equation}
	with respect to the $D^{2}$-compatible Weitzenb\"{o}ck covariant derivative (Definition~\ref{def: P_compatible_connection}) $\connectionEndPiE_{X_{\fg/2}}$ along the geodesic vector field $X_{\fg/2}$  (Remark~\ref{rem: HVF_canonical_relation}). 
	\newline 

	If $\big( \sE \to \sM, \symb{D}, \tilde{\nabla} \big)$ is a Dirac module bundle with the corresponding Dirac-type operator $\tilde{D}$ then~\eqref{eq: symbol_causal_propagator_Dirac_covarinatly_constant} holds with the replacement of $\connectionEndPiE$ by the $\tilde{D}^{2}$-compatible connection $\tilde{\nabla}^{\ms \pi^{*} \Hom{\sE, \sE}}$ induced by the Clifford connection $\tilde{\nabla}$.        
\end{lemma}
%
%
%
%
%
%
%
%
%
%
\subsection{Cauchy problem}
\label{sec: Cauchy_problem_Dirac}
Since $D^{2}$ is a normally hyperbolic operator and $\spacetime$ is a globally hyperbolic spacetime, the Cauchy problem for $D$ is well-posed~\cite[Thm. 2]{Muehlhoff_JMP_2011}. 
In other words, for an arbitrary but fixed Cauchy hypersurface $\varSigma_{t} \subset \sM$, the mapping  
\begin{equation} \label{eq: sol_Dirac_Cauchy_data}
	\cR_{t} : \ker D \to C_{\mathrm{c}}^{\infty} (\varSigma_{t}; \sE_{\varSigma_{t}}), ~ u \mapsto \cR_{t} (u) := u|_{\varSigma_{t}}, 
	\quad \supp{u} \subset J \big( \supp (u|_{\varSigma_{t}}) \big)  
\end{equation}
is a homeomorphism, where $J$ is as defined in Section~\ref{sec: convention}. 
Employing this topological isomorphism and Assumption~\ref{asp: trace_formula}~\ref{asp: hermitian_form_Dirac_op_SST}, $\ker D$ can be equipped with a (positive definite) hermitian inner product  
(see e.g.~\cite[Lem. 3.17]{Baer_Springer_2012}):   
\begin{equation} \label{eq: def_hermitian_form_Dirac_type_op}
	\langle u | v \rangle := \int_{\varSigma} \Big( \fc \big( (x, \zeta) \otimes u|_{\varSigma} \big) \big| v |_{\varSigma} \Big)_{x} \dVolh (x)
\end{equation}
where $\dVolh$ is the Riemannian volume element on $(\varSigma, \fh)$ and $(\cdot, \zeta)$ as in Assumption~\ref{asp: trace_formula}~\ref{asp: hermitian_form_Dirac_op_SST}.  
By global hyperbolicity (Section~\ref{sec: GHSST}) of $\spacetime$, there exists a (non-unique) global Cauchy temporal function $\boldsymbol{t}$ such that each Cauchy hypersurface $\varSigma_{t} := \boldsymbol{t}^{-1} (t)$ for any $t \in \R$ is a level set of $\boldsymbol{t}$. 
Then $\zeta := \rd \boldsymbol{t} / \| \rd \boldsymbol{t} \|$ is a unit normal covector field on $\sM$ along $\varSigma_{t}$ and we choose the time-orientation employing the global timelike (Killing) vector field $Z$ such that $\zeta$ is future-directed. 
On a static spacetime, we can choose the Killing covector field $(\partial_{t})^{\flat}$ (up to normalisation) for $\zeta$.   

We remark that $\scalarProdTwo{u}{v}$ is independent of chosen Cauchy hypersurface $\varSigma$ due to the Green-Stokes formula.  
Thus, $(\sE \to \sM, \langle \cdot|\cdot \rangle)$ is a hermitian vector bundle where the hermitian form $\langle \cdot|\cdot \rangle$, of course depends on $(\cdot, \zeta)$ but is independent of chosen Cauchy hypersurface.  

The preceding equation entails that, given an initial data $k \in C_{\mathrm{c}}^{\infty} (\varSigma; \sE_{\varSigma})$, any smooth solution $u$ of the Dirac equation can be expressed as 
\begin{equation} \label{eq: smooth_sol_Dirac_op_Cauchy_data}
	(u|v) = \big( - \ri \, G \circ (\iota_{\ms \varSigma}^{*})^{-1} \, \fc (\zeta  \otimes k) \big| v \big), 
\end{equation} 
where $\iota_{\ms \varSigma}^{*} : \comSecsME \to C_{\mathrm{c}}^{\infty} (\varSigma; \sE_{\varSigma})$ is the restriction operator (discussed elaborately in Appendix~\ref{sec: restriction_op}). 
All the maps in the exact complex  
\begin{equation} \label{eq: exact_sequence_causal_propagator_Dirac}
	0 
	\to \comSecsME
	\xrightarrow{~ D ~} \comSecsME 
	\xrightarrow{~ G ~} C_{\mathrm{sc}}^{\infty} (\sM; \sE)  
	\xrightarrow{~ D ~} C_{\mathrm{sc}}^{\infty} (\sM; \sE) 
\end{equation}
are sequentially continuous as a consequence of $D$ being a local operator and the following exact complex being sequentially continuous  
(see e.g.~\cite[Prop. 3.4.8]{Baer_EMS_2007})  
\begin{equation}
	0 
	\to \comSecsME
	\xrightarrow{~ D^{2} ~} \comSecsME 
	\xrightarrow{~ F ~} C_{\mathrm{sc}}^{\infty} (\sM; \sE)  
	\xrightarrow{~ D^{2} ~} C_{\mathrm{sc}}^{\infty} (\sM; \sE)   
\end{equation}
where $F$ resp. $G$ are causal propagators for $D^{2}$ resp. $D$. 
%
%
%
%
%
%
%
%
%
%
\subsection{Cauchy evolution map}
This mapping is defined by 
\begin{equation} \label{eq: def_Cauchy_evolution_op}
	U_{t', t} := \cR_{t} \circ (\cR_{t'})^{-1} : C_{\mathrm{c}}^{\infty} \big( \varSigma_{t'}; \sE_{\varSigma_{t'}} \big) \to C_{\mathrm{c}}^{\infty} \big( \varSigma_{t}; \sE_{\varSigma_{t}} \big),     
\end{equation} 
which is a homeomorphism and extends to a unitary operator (denoted by the same symbol) on the space of square integrable sections on Cauchy hypersurfaces, i.e.,  
\begin{equation}
	U_{t', t} : L^{2} \big( \varSigma_{t'}; \sE_{\varSigma_{t'}} \big) \to L^{2} (\varSigma_{t}; \sE_{\varSigma_{t}})
\end{equation}
is an isometry.  
%
%
%
%
%
%
%
%
%
%
\section{Trace formula} 
\label{sec: trace_Cauchy_evolution_op}
In this section we will work on the set-up in Section~\ref{sec: result}. 
That is, $\varSigma \subset \sM$ is an embedded submanifold and $\iota_{\ms \varSigma} : \varSigma \hookrightarrow \sM$ is proper. 
Consequently, $\iota_{\ms \varSigma}^{*} : \comSecsME \to C_{\mathrm{c}}^{\infty} (\varSigma; \sE_{\varSigma}) = \secECauchy$. 
%
%
%
\begin{theorem} \label{thm: trace_Cauchy_evolution_op}
	As in the terminologies of Theorem~\ref{thm: trace_formula_t_zero}, let $\fG$ be the Schwartz kernel of the causal propagator for $D$. 
	The smoothed-out time evolution operator $U_{\rho}$ given by~\eqref{eq: def_smoothed_time_evolution_op}, is a trace-class operator on the Hilbert space $\sH := (\ker D,~\eqref{eq: def_hermitian_form_Dirac_type_op})$ and its trace is given by 
	\begin{equation} \label{eq: trace_Cauchy_evolution_op}
		\Tr U_{\rho} 
		= 
		- \ri \int_{\varSigma} \tr \int_{\R} \big( \varXi_{-t}^{*} \circ \fG \circ (\iota_{\ms \varSigma}^{*})^{-1} \, \symb{D} (\cdot, \zeta ) \big) (x, y) \, (\cF^{-1} \rho) (t) \, \rd t |_{x = y} \, \dVolh (x) 
	\end{equation}
	where $\varXi_{t}^{*} : \comSecsME \to \comSecsME$ is the induced Killing flow and $\iota^{*} : \comSecsME \to \secECauchy$ is the restriction operator. 
	In this article, $\varXi_{-t}^{*} \fG$ is meant to be the pullback of $\fG (\cdot, )$ in the first argument via $\varXi_{-t}$.   
\end{theorem}
%
%
%
\begin{proof}
	By Assumption~\ref{asp: trace_formula}, the hermitian form~\eqref{eq: def_hermitian_form_Dirac_type_op} and the retarded (resp. advanced) $\retGreenOp$ (resp. $\advGreenOp$) propagators are preserved under the action of $\varXi_{t}^{*}$. 
	Hence $G$ is preserved as well.  
	In other words, if $u \in \ker D$ having initial data on some $\varSigma$ then $\varXi_{t}^{*} u \in \ker D$ having Cauchy data on some $\varSigma_{t}$, which means that the time flow $\varXi_{t}$ induces time evolution of Cauchy data.   
	Let us now choose an arbitrary but fixed $\varSigma$. 
	This picks a global time coordinate $t$ on $\sM$ and then the generator of the Killing flow $\varXi_{t}$ is given by $\partial_{t}$. 
	Thus, $U_{t}$ is identified with the evolution of Cauchy data via the induced Killing flow (see Figure~\ref{fig: time_evolution_causal_propagator_restriction_op_Killing_flow} for a schematic illustration): 
	\begin{equation} \label{eq: U_t_Cauchy_restriction_op_Killing_flow}
		U_{t} = \cR \circ \varXi_{-t}^{*} \circ \cR^{-1}.
	\end{equation} 
	We read off $\cR^{-1} = - \ri \, G \circ (\iota_{\ms \varSigma}^{*})^{-1} \symb{D} (\cdot, \zeta )$ from~\eqref{eq: smooth_sol_Dirac_op_Cauchy_data} and observe that the twisted wavefront set~\eqref{eq: def_geodesic_relation} of $\fG$ contains only lightlike covectors. 
	Therefore, integration over $t$ results a smooth Schwartz kernel $\fU_{\rho} (x, y)$ and the expression of $\Tr U_{\rho}$ entails from~\eqref{eq: def_tr_smoothed_Cauchy_evolution_op}. 
\end{proof}
%
%
%
%
%
%
%
%
%
%
\section{Spectral theory of $L$ on $\ker D$}
\label{sec: spectral_theory_L}
Recall that, by Stone's theorem, every strongly continuous one parameter unitary group $\{ U_{t} \}_{t \in \R}$ on a Hilbert space $\cH$ has a unique generator $A$, i.e., $U_{t} = \re^{- \ri t A}$. 
If 
\begin{equation}
	U_{\rho} := \int_{\R} \re^{- \ri t A} (\cF^{-1} \rho) (t) \, \rd t 
\end{equation}
is a compact operator for any Schwartz function $\rho$ on $\R$ such that $\supp (\cF^{-1} \rho)$ is compact, then the spectrum of $A$ is discrete and consists of eigenvalues $\lambda_{n}$ of finite algebraic multiplicities $m_{n}$. 
Moreover, whenever $U_{\rho}$ is trace-class then 
\begin{equation}
	\Tr U_{\rho} = \sum_{n} m_{n} \, \rho (\lambda_{n}).   
\end{equation}
%
%
%
\begin{proof}[Proof of Theorem~\ref{thm: spectrum_L}]
	Since $D$ commutes with induced Killing flow $\varXi_{t}^{*}$ for all $t \in \R$ and $U_{t}$ is the time evolution operator of $D$, the discreteness of $\spec L$ and the polynomial growth of eigenvalues $\lambda_{n}$ follow from the fact that $U_{\rho} := \int_{\R} \exp (\ri t L) \, (\cF^{-1} \rho) (t) \, \rd t$ is a trace-class operator for any Schwartz function $\rho$ on $\R$ such that $\supp (\cF^{-1} \rho)$ is compact. 
	All $\lambda_{n}$ are real as a consequence of the selfadjointness of $L$ on $\ker D$. 
\end{proof}
%
%
%
%
%
%
%
%
%
%
\section{Proof of Theorem~\ref{thm: trace_formula_t_zero} and~\ref{thm: trace_formula_t_T}}
\label{sec: proof_trace_formula_nondegenerate_fixed_pt}
In order to implement the strategy outlined in Section~\ref{sec: proof_strategy}, we recall that that
\begin{equation}
	\uppi_{*} \circ \varDelta^{*} : 
	C_{\mathrm{c}}^{\infty} \big( \R \times \varSigma \times \varSigma; \Hom{\sE_{\varSigma}, \sE_{\varSigma}} \big) 
	\to 
	C^{\infty} (\R; \End \sE_{\varSigma})
\end{equation}
is a zero-order Fourier integral operator whose Schwartz kernel $\fK$ is that of an identity map~\cite[Lem. 5.2]{Sandoval_CPDE_1999} 
(see also~\cite[Lem. 6.2, 6.3]{Duistermaat_Inventmath_1975}):   
\begin{subequations}
	\begin{eqnarray}
		&& 
		\fK \in I^{0} \Big( \R \times \R \times \varSigma \times \varSigma, C_{\uppi_{*} \circ \varDelta^{*}}; \mathrm{Hom} \big( \Hom{\sE_{\varSigma}, \sE_{\varSigma}}, \End \sE_{\varSigma} \big) \Big), 
		\\ 
		&& 
		C_{\uppi_{*} \circ \varDelta^{*}} = \big( \varDelta (\R \times \varSigma)  \big)^{\perp *}, 
		\\ 
		&& 
		\symb{\fK} = \Pi^{*} \big( |\rd t \wedge \rd \tau \wedge \rd x \wedge \rd \xi|^{1/2} \big) \one. 
	\end{eqnarray}
\end{subequations}
Here the canonical relation $C_{\uppi_{*} \circ \varDelta^{*}} := \{ (t, \tau; t, \tau; x, \xi; x, - \xi) \in \dotCoTan \R \times \dotCoTan \R \times \dotCoTanCauchy \times \dotCoTanCauchy \}$ of $\uppi_{*} \circ \varDelta^{*}$ has been identified with the cornormal bundle $\big( \varDelta (\R \times \varSigma)  \big)^{\perp *}$ to the diagonal in $\R \times \varSigma \times \R \times \varSigma$, $\Pi$ is the projector 
$
C_{\pi_{*} \varDelta^{*}} \ni (t, \tau; x, \xi; x, - \xi; t, - \tau) \mapsto (t, \tau; x, \xi) \in \dotCoTan \R \times \dotCoTanCauchy 
$, 
and the Keller-Maslov bundle $\T \to \big( \varDelta (\R \times \varSigma)  \big)^{\perp *}$ is trivial. 

Theorem~\ref{thm: trace_Cauchy_evolution_op} implies that the trace $\Tr U_{t}$ exists as a distribution in $\cD' (\R)$ and it can be re-expressed as: 
\begin{equation} \label{eq: tr_U_t}
	\Tr U_{t} = \int_{\varSigma} \tr \big( (\varXi_{-t}^{*} \fG) (x, y) \, \symb{D} (y, \zeta ) \big) \big|_{x = y} \, \dVolh (x). 
\end{equation}
Since the geodesic relation is disjoint with the conormal bundle $\varSigma^{\perp *}$, restriction of $\varXi_{-t}^{*} \fG$ is well-defined. 
Then applying~\eqref{eq: Tr_U_t_tr_pushforward_pullback}, the preceding equation can be written as 
\begin{equation} \label{eq: Tr_U_t_tr_pushforward_pullback_restriction_causal_propagator}
	\Tr U_{t} 
	=  
	\tr \Big( \uppi_{*} \circ \varDelta^{*} \circ (\iota_{x}^{*} \boxtimes \iota_{y}^{*}) \big( (\varXi_{-t}^{*} \fG) (x, y) \, \symb{D} (y, \zeta) \big) \Big) 
\end{equation}
for any $t \in \R$ and any $x, y \in \sM$. 
Hence, our task is to compute the principal symbol of 
\begin{equation} \label{eq: def_G_t}
	\fG_{t} := (\iota_{x}^{*} \boxtimes \iota_{y}^{*}) \big( (\varXi_{-t}^{*} \fG) (x, y) \big). 
\end{equation} 
We have worked out $\iota_{\ms \varSigma}^{*}$ in Appendix~\ref{sec: restriction_op}, so let us begin with by describing $\varXi_{-t}^{*} \fG$ as a Lagrangian distribution.  
To begin with one notes that $w$ appearing in the expression of $\symb{\fG}$ satisfies the differential equation~\cite[Thm. 6.6.1]{Duistermaat_ActaMath_1972} 
\begin{equation}
	(X_{\fg^{\pm}/2} \pm \ri \subSymb{D^{2}, \pm}) w = 0, 
\end{equation}
by Lemma~\ref{lem: causal_propagator_Dirac} and Definition~\ref{def: P_compatible_connection}, where $\fg^{\pm}$ resp. $\subSymb{\square, \pm}$ are the lifts of $\fg^{-1}$ resp. $\subSymb{\square}$ to $\dotCoTansM \times \dotCoTansM$ via the projections on the first resp. second copies of $\dotCoTansM$. 
This means that $\ri (2 \pi)^{3/4} \symb{D} \circ w \sqrt{|\rd t|} \otimes \sqrt{|\dVol_{\ms C}|} / 2$ is the principal symbol of $\varXi^{*} \fG$ on each $C^{+}$ and $C^{-}$. 
Employing~\eqref{eq: pullback_kernel},~\eqref{eq: pullback_Killing_flow_Lagrangian_dist}, Lemma~\ref{lem: causal_propagator_Dirac}  
and Appendix~\ref{sec: composition_halfdensity_Maslov_vector_bundle_canonical_relation}, it is then  straightforward to obtain 
%
%
%
\begin{lemma} \label{lem: pullback_Killing_flow_causal_propagator}
	Suppose that $(\sE \to \sM, \symb{D})$ is a bundle of Clifford modules over a globally hyperbolic stationary spacetime $\SST$ and that $\fG$ is the Schwartz kernel of the causal propagator of any Dirac-type operator $D$ on $\sE$.   
	Then the pullback of $\fG$ by the Killing flow $(\varXi, \hat{\varXi})$,     
	\begin{subequations} \label{eq: pullback_Killing_flow_causal_propagator_kernel}
		\begin{eqnarray}
			&& 
			\varXi^{*} \fG   
			\in 
			I^{-3/4} \big( \R \times \sM \times \sM, \varGamma' \circ C; \Hom{\sE, \varXi^{*} \sE} \big), 
			\label{eq: pullback_Killing_flow_causal_propagator_Lagrangian_dist}
			\\ 
			&& 
			\varGamma' \circ C  
			= 
			\big\{ \big( t, - \xi (Z); x, \xi; y, \eta \big) \in \dotCoTan \R \times \coLightBun \times \coLightBun \,|\, \xxi = \coTan \! \varXi_{t} \circ \varPhi_{s} \yeta \big\}, \hspace*{1.25cm} 
			\label{eq: def_canonical_relation_pullback_Killing_flow_causal_propagator}
			\\ 
			&& 
			\symb{\varXi^{*} \fG}  
			\equiv  
			\ri (2 \pi)^{\frac{3}{4}} \symb{D} \circ w \big( \coTan \! \varXi_{t} \circ \varPhi_{s} \yeta, \yeta \big) \, |\rd t|^{\frac{1}{2}} \otimes \big| \dVol_{\ms C} \big( \coTan \! \varXi_{t} \circ \varPhi_{s} \yeta, \yeta \big) \big|^{\frac{1}{2}}, \qquad \quad  
			\label{eq: symbol_pullback_Killing_flow_causal_propagator}
		\end{eqnarray}
	\end{subequations}
	where the principal symbol is modulo Keller-Maslov part, $Z$ is the infinitesimal generator of $\varXi$ and all other symbols are as defined in Lemma~\ref{lem: causal_propagator_Dirac} and~\eqref{eq: def_canonical_relation_pullback}. 
\end{lemma}
%
%
%
Next, we compute the restriction of $\varXi^{*} \fG$ on $\varSigma_{x} \times \varSigma_{y}$ by an application of Appendix~\ref{sec: restriction_op}.  
%
%
%
\begin{lemma} \label{lem: restriction_pullback_causal_propagator}
	As in the terminologies of Lemma~\ref{lem: pullback_Killing_flow_causal_propagator}, let $\varSigma$ be a Cauchy hypersurface of $\sM$. 
	Then the distribution~\eqref{eq: def_G_t} 
	\begin{subequations}
		\begin{eqnarray}
			\fG_{t} 
			& \in &   
			I^{-1/4} \Big( \R \times \varSigma \times \varSigma, \cC_{t}; \mathrm{Hom} \big( \sE_{\varSigma}, (\varXi^{*} \sE)_{\varSigma} \big) \Big), 
			\\ 
			\cC_{t} 
			& := &  
			\big\{ \big( (t, - \xi (Z)), \xxi |_{\tangent \varSigma}, \yeta |_{\tangent \varSigma} \big) \in \dotCoTan \R \times \coTansM_{\varSigma} \times \coTansM_{\varSigma} \,|\, 
			\nonumber \\ 
			&& 
			\xxi = \coTan \! \varXi_{t} \circ \varPhi_{s} \yeta \big\}, 
			\label{eq: def_canonical_relation_Cauchy_evolution_op}
			\\ 
			\symb{\fG_{t}} 
			& \equiv &   
			\ri (2 \pi)^{\frac{1}{4}}     
			\symb{D} \circ w \big( \coTan \! \varXi_{t} \circ \varPhi_{s} \yeta |_{\tangent \varSigma}, \yeta |_{\tangent \varSigma} \big) 
			\nonumber \\ 
			&& 
			|\rd t|^{\frac{1}{2}} \otimes \big| \dVol_{\ms C} \big( \coTan \! \varXi_{t} \circ \varPhi_{s} \yeta |_{\tangent \varSigma}, \yeta |_{\tangent \varSigma} \big) \big|^{\frac{1}{2}}, 
			\label{eq: symbol_restriction_pullback_causal_propagator}
		\end{eqnarray}
	\end{subequations}
	where the principal symbol is modulo the Keller-Maslov part.    
\end{lemma}
%
%
%
\begin{proof}
	The first two assertions are immediate 
	from~\eqref{eq: restriction_kernel_Lagrangian_dist},~\eqref{eq: def_canonical_relation_restriction},~\eqref{eq: pullback_Killing_flow_causal_propagator_Lagrangian_dist} 
	and~\eqref{eq: def_canonical_relation_pullback_Killing_flow_causal_propagator}. 
	To compute the principal symbol one notes that $\dVol_{\ms \coLightBun} \yeta$ induces a volume element $\dVol_{\ms \coLightBun_{\varSigma}} \yeta := \dVol_{\ms \coLightBun} / \rd y^{1}$ on $\coLightBun_{\varSigma}$ when $\varSigma$ is parametrised by $y^{1} = \cst$ and then $\dVol_{\ms \coLightBun_{\varSigma}} \yeta = \rd \eta_{1} \wedge \rd y' \wedge \rd \eta'$ in the adapted coordinates $(y^{1} = \cst, y', \eta_{1}, \eta')$ on $\coTansM$. 
	Finally, the claim follows from~\eqref{eq: symbol_restriction},~\eqref{eq: symbol_pullback_Killing_flow_causal_propagator},~\eqref{eq: symbol_restriction} and transporting $\yeta$ to $\xxi$ by the geodesic flow and the Killing flow. 
\end{proof}
%
%
%

If the composition $C'_{\uppi_{*} \varDelta^{*}} \circ \cC_{t}$ is clean then we can compute $\Tr U_{t}$ by the standard composition of Fourier integral operators. 
But, in general, 
\begin{equation} 
	C'_{\uppi_{*} \varDelta^{*}} \circ \cC_{t} 
	= 
	\big\{ (t, \tau) \in \R \times \dot{\R} \,|\, \tau = - \xi (Z), \xxi |_{\tangent \varSigma} = \big( \coTan \varXi_{-t} \circ \varPhi_{t} \yeta \big) |_{\tangent \varSigma} \big\}  
\end{equation}
may not be clean because
(see Appendix~\eqref{eq: def_C_star_varLambda},~\eqref{eq: def_fibre_projection_clean_composition} and~\eqref{eq: varPi} for the symbol ${\scriptsize \text{\FiveStarOpen}}$ and further details)
the fibres of $C'_{\uppi_{*} \varDelta^{*}} {\scriptsize \text{\FiveStarOpen}} \cC_{t} \to C'_{\pi_{*} \varDelta^{*}} \circ \cC_{t}$ can be identified with the fibres over $\tau \in \dotCoTan_{t} \R$, i.e., the set $\{\kF_{t}\}$~\eqref{eq: def_F_T} of periodic geodesics, where we have used~\eqref{eq: def_Hamiltonian_scaled_lightlike_geodesic} and~\eqref{eq: symplectic_diffeo_lightlike_geodesic_coTanCauchy}.   
Then, even if $\{ \kF_{t} \}$ happen to be manifolds, the chances of $\dim \kF_{t}$ is a constant for all $t$ is very low; for instance, if all the orbits of $\varXi_{t}^{\cN}$ are periodic with the same period $t = T$, then the composition is clean  
(see e.g.~\cite[p. 289]{Meinrenken_ReptMathPhys_1992}). 
We remark that this, however, is not an issue for the trivial period ($T=0$) and the assumptions made on classical dynamics in Theorem~\ref{thm: trace_formula_t_T} ensures a clean intersection in the case of non-trivial periods ($T \neq 0$). 

Let us record for the future computations that  
\begin{eqnarray}
	&& 
	\dim \sM = d, \dim \coTansM = 2d, \dim \coLightBun = 2d-1, \dim \cN = 2d-2, \dim \tilde{\cN} = 2d-3, 
	\nonumber \\ 
	&& 
	\dim \varSigma = d-1, \dim \coTanCauchy = 2d-2 = \dim \coLightBun_{\varSigma}. 
\end{eqnarray} 
%
%
%
%
%
%
%
%
%
%
\subsection{Principal symbol at $T = 0$}
\label{sec: symbol_t_zero}
We begin with the trivial periodic orbits where a big singularity is expected as $U_{t}$ reduces to an identity operator in this situation. 
To describe $\symb{\Tr U_{0}}$ it is useful to have the notion of the 
symplectic residue~\cite[Def. 6.1]{Guillemin_AdvMath_1985},  
introduced by Guillemin to derive Weyl's law in the context of Weyl algebra quantising a conic symplectic manifold. 
By construction (cf.~\eqref{eq: def_Hamiltonian_scaled_lightlike_geodesic}), $H^{1-d}$ is a homogeneous function of degree $1-d$. 
Then its symplectic residue is defined by 
$\res H^{1-d} := \int_{\tilde{\cN}} \vartheta_{H^{1-d}}$, 
where $\vartheta$ is the interior multiplication of the symplectic volume form $\dVol_{\ms \cN}$ on $\cN$ by the Euler vector field (Section~\ref{sec: classical_dyanmics}) on it and $\vartheta_{H^{1-d}}$ is the pull-back of $H^{1-d} \vartheta$ on $\cN$. 
Homogeneity of $H$ entails 
(cf.~\cite[Proof of Lemma 6.3]{Guillemin_AdvMath_1985}) 
\begin{equation} \label{eq: symplectic_residue_volume_lightlike_geodesic}
	\res H^{1-d} = (d-1) \, \vol (\cN_{H \leq 1}). 
\end{equation}
On a standard stationary spacetime $\vol (\cN_{H \leq 1})$ has been computed by  
Strohmaier-Zelditch~\cite[$(15)$]{Strohmaier_AdvMath_2021} 
and it is given by~\eqref{eq: vol_energy_surface}.  
%
%
%
\begin{lemma} \label{lem: symbol_Tr_U_0}
	As in the terminologies of Theorem~\ref{thm: trace_Cauchy_evolution_op} (and hence Theorem~\ref{thm: trace_formula_t_zero} as well), 
	$\fU_{t} (x, y) \in I^{-1/4} \Big( \R \times \varSigma \times \varSigma, \cC_{t}; \mathrm{Hom} \big( \sE_{\varSigma}, (\varXi^{*} \sE)_{\varSigma} \big) \Big) $ associated with the canonical relation $\cC_{t}$~\eqref{eq: def_canonical_relation_Cauchy_evolution_op} and, modulo the Keller-Maslov part, its principal symbol is given by  
	\begin{equation}
		\symb{\fU_{t}} \xxiyeta \equiv \symb{\fG_{t}} \xxiyeta \, \symb{D} \yeta, 
	\end{equation}
	where $\symb{\fG_{t}}$ is given by~\eqref{eq: symbol_restriction_pullback_causal_propagator}. 

	Furthermore, $\Tr U_{t}$ is a Lagrangian distribution on $\R$ of order $d - 7/4$ associated with the Lagrangian submanifold $\Lambda_{T}$ and its principal symbol at $T = 0$ is given by  
	\begin{equation}
		\symb{\Tr U_{0}} (\tau) \equiv r \frac{d-1}{(2 \pi)^{d-1}} \vol (\cN_{H \leq 1}) \, |\tau|^{d-2} \sqrt{|\rd \tau|}. 
	\end{equation}
\end{lemma}
%
%
%
\begin{proof}
	The first assertion simply follows from $\fU_{t} (x, y) = \fG_{t} (x, y) \, \symb{D} (y, \zeta)$.
	We then use identification of $\cN$ with $\coLightBun_{\varSigma}$ so that $|\dVol_{\ms \cN}|$ is identified with $|\dVol_{\ms \coLightBun_{\varSigma}}|$. 
	If $\tE$ is a regular value of $H$ then $\cN_{\tE}$ is a codimension one $X_{H} |_{\cN_{\tE}}$ invariant embedded submanifold of $\cN$, which inherits a natural volume form $\dVol_{\ms \cN_{\tE}}$ from $\dVol_{\ms \cN}$, invariant under the action of $X_{H} |_{\cN_{\tE}}$~\cite[Thm. 3.4.12]{Abraham_AMS_1978}: 
	\begin{equation}
		\dVol_{\! \ms \cN} = \dVol_{\! \ms \cN_{\tE}} \wedge \frac{\rd H}{\| \grad H \|} 
		\Leftrightarrow 
		\dVol_{\! \ms \cN_{\tE}} (\ldots) = \dVol_{\! \ms \cN} \bigg( \frac{\grad H}{\| \grad H \|}, \ldots \bigg). 
	\end{equation}

	One observes that at $T = 0$, 
	$\coTan \varXi_{0} \big( \varPhi_{0} \yeta \big) = \yeta$ 
	and 
	$C = \varDelta \, \coLightBun$ 
	so that 
	$\cC'_{0} = \big\{ \big( 0, - \xi (Z) \big) \big\} \times (\varDelta \, \coLightBun) |_{\tangent \varSigma \times \tangent \varSigma}$ 
	and 
	$\big| \dVol_{\ms C} (y, \eta; y, \eta) |_{\tangent \varSigma \times \tangent \varSigma} \big|^{1/2} = |\dVol_{\ms \cN} \yeta|$. 
	The excess of $C'_{\uppi_{*} \varDelta^{*}} \circ \cC_{0}$ is (see Appendix~\eqref{eq: excess_fibre_projection_clean_composition}) $e = \dim \kF_{0} = 2d - 3$.  
	Hence, by the composition of Lagrangian distributions, 
	$\Tr U_{0} \in I^{d - 7/4} (\R, \Lambda_{T})$. 
	One observes~\cite[$(6.6)$]{Duistermaat_Inventmath_1975} 
	that $\varphi (t, \tau)$ defined by $(t - T) \tau$ resp. $0$ for $\tau < 0$ resp. $\tau \geq 0$ is a phase function for $\Lambda_{T}$. 
	Obviously, $\varphi$ is not smooth around $\tau = 0$ but the required modification will only reflects by some smooth terms in the oscillatory integral described below, which we suppress notationally for brevity. 
	Any element of this Lagrangian distribution can be written as scalar multiples $\rc_{1}, \rc_{2}, \ldots$ of an oscillatory integral of the form 
	\begin{equation}
		(\Tr U_{0}) (t) = (2 \pi)^{-3/4} (2 \pi)^{d-1} \int_{\R_{\geq 0}} \re^{-\ri t \tau} ( \rc_{1} \tau^{d-2} + \rc_{2} \tau^{d-3} + \ldots) \rd \tau
	\end{equation}
	because $\symb{\Tr U_{0}}$ must be of order $d-2$ as entailed by the order of $\Tr U_{0}$. 

	Computation of $\symb{\Tr U_{0}}$ involves the following four steps as detailed in Section~\ref{sec: composition_halfdensity_Maslov_vector_bundle_canonical_relation}.  
	Briefly speaking, first one obtains the tensor product $\symb{\fK} \otimes \symb{\fU_{0}}$ followed by intersecting with the diagonal. 
	Then we integrate over $\kF_{0}$ and take $\tr$ of the hindmost quantity. 
	Putting all these together and comparing with the last expression of $\Tr U_{0}$ we obtain  
	\begin{equation}
		\symb{\Tr U_{0}} (0, \tau) = r \frac{\res (H^{1-d})}{(2 \pi)^{d-1}} |\tau|^{d-2} \sqrt{|\rd \tau|}, 
	\end{equation}
	where we have used Lemma~\ref{lem: restriction_pullback_causal_propagator},~\ref{lem: causal_propagator_Dirac} and the identity  
	\begin{equation}
		\tr \big( \symb{D} \yeta \, \symb{D} (y, \zeta) \big) 
		= 
		\fg_{y}^{-1} (\eta, \zeta) \, \rk (\sE)
	\end{equation} 
	using the cyclicity of trace and~\eqref{eq: anticommutation_symbol_Dirac_type}. 
\end{proof}
%
%
%
%
%
%
%
%
%
%
\subsection{Principal symbol at $t = T \in \cP$} 
\label{sec: symbol_t_T}
This is the scenario corresponding to the non-trivial periodic orbits when $C_{\uppi_{*} \varDelta^{*}} \circ \cC_{\ms T}$ is clean or equivalently the set of fixed points $\kF_{T}$ is clean.  
Recall, the fixed point sets of $\varXi_{T}^{\cN}$ are the union of periodic orbits $\gamma$ of $\varXi_{T}^{\cN}$ which depends substantially on manifold $\spacetime$ and we are only concerned about non-degenerate orbits as briefly defined below 
(details are available  
in~\cite[Sec. 4]{Duistermaat_Inventmath_1975}).   
Let $\upsigma$ is the symplectic form on $\coTansM$. 
The induced Killing flow $\varXi_{s}^{\cN}$ preserves the level sets $\{ H (c) = \tau \}$ since $\varXi_{s}^{\cN}$ is the Hamiltonian $H$~\eqref{eq: def_Hamiltonian_scaled_lightlike_geodesic} flow on $\cN$. 
Then its restriction to $\cV := \tangent_{c} \{ H = \tE \}$ for any $\tE$ satisfies $\upsigma |_{\cV} (X_{H}, \cdot) = 0$ and by the 
\textit{non-degeneracy}~\cite[p. 44]{Strohmaier_AdvMath_2021} 
of $\gamma$, it is meant that the \textit{nullspace of $\upsigma |_{\cV}$ is spanned by} $X_{H}$. 
This condition implies that 
$\ker \big( I - (\rd_{\gamma} \varXi_{s}^{\cV}) = \R X_{H}$ 
and 
$\mathrm{img} \big( I - (\rd_{\gamma} \varXi_{s}^{\cV}) \big) = \cV / \R X_{H}$,  
where $\rd_{\gamma} \varXi_{s}^{\cN} : \cV \to \cV$ and $\ker (I - \rd_{\gamma} \varXi_{s}^{\cN})$ is to be understood as the kernel of $I - \rd_{\gamma} \varXi_{s}^{\cN}$ on $\cV$. 
In this setting, the linearised Poincar\'{e} map $\rP_{\gamma}$ is defined as   
\begin{equation}
	I - \rP_{\gamma} : \cV / \ker (I - \rd_{\gamma} \varXi_{s}^{\cN}) \to \cV / \ker (I - \rd_{\gamma} \varXi_{s}^{\cN}) 
\end{equation} 
the linear symplectic quotient map induced by $I - \rd_{\gamma} \varXi_{s}^{\cN}$. 
%
%
%
\begin{lemma}
	As in the terminologies of Theorem~\ref{thm: trace_formula_t_T}, the principal symbol of $\Tr U_{t}$ at $t = T \in \cP$ is   
	\begin{equation}
		\symb{\Tr U_{T}} (\tau) =   
		\frac{\tau^{d-2}}{(2 \pi)^{d-1}} \int_{\kF_{T}} 
		\tr \Big( \symb{D} (\gamma) \, \cT_{\varXi_{\ms T}^{\ms \cN} (\gamma)} \symb{D} (x, \zeta) \Big) 
		\frac{\re^{- \ri \pi \km (\gamma) / 2} |\rd T| \otimes \sqrt{|\rd \tau|}}{\sqrt{|\det (I - \rP_{\gamma})|}}. 
	\end{equation}
\end{lemma}
%
%
%
\begin{proof}
	We begin with the fact that non-degenerate $\gamma$ belongs to a $2$-dimensional cylinder, transversally intersecting the energy hypersurfaces~\cite[p. 576]{Abraham_AMS_1978} which implies that one can use $\tau = - H (\gamma)$ as a coordinate for fibre over $(T, \tau)$. 
	Then the half-density valued density at $(T, \tau)$ is given by~\cite[pp. 60 - 61]{Duistermaat_Inventmath_1975}  
	(see also, e.g.~\cite[Thm. 6.1.1]{Guillemin_InternationalP_2013})
	\begin{equation}
		|\rd T| \otimes \sqrt{\frac{|\rd \tau|}{|\det (I - \rP_{\gamma})|}}, 
	\end{equation}
	where $|\rd T|$ is the density on $\gamma$ induced by the flow.  

	To compute $\symb{\Tr \fU_{T}}$ we use identifications $\cN \equiv \coLightBun_{\varSigma}, |\dVol_{\ms \cN}| \equiv |\dVol_{\ms \coLightBun_{\varSigma}}|$ once again. 
	For $T \neq 0$, $C$ is no more given by $\varDelta \, \coLightBun$ and so we read off the off-diagonal expression of $\symb{\fG}$ from Lemma~\ref{lem: causal_propagator_Dirac} to obtain   
	\begin{equation}
		\symb{\fU_{T}} \xxiyeta 
		\equiv 
		(2 \pi)^{\frac{1}{4}} \ri     
		\symb{D} \xxi \, w \xxiyeta  
		|\rd t|^{\frac{1}{2}} \otimes \big| \dVol_{\ms C} \xxiyeta \big|^{\frac{1}{2}} 
	\end{equation}
	modulo the Keller-Maslov contribution, where $\xxi \equiv \coTan \! \varXi_{t} \circ \varPhi_{s} (y^{\backprime}, \eta^{\backprime}) |_{\tangent \varSigma}, \yeta \equiv (y^{\backprime}, \eta^{\backprime}) |_{\tangent \varSigma}$ for $(y^{\backprime}, \eta^{\backprime}) \in \coLightBun$. 
	As before, $\symb{\Tr U_{T}}$ is computed employing~\eqref{eq: Tr_U_t_tr_pushforward_pullback} by performing the four steps used in Lemma~\ref{lem: symbol_Tr_U_0}. 
	These yield the claimed expression, modulo the contribution coming from the Keller-Maslov line bundle $\Maslov \to \Lambda_{T}$. 

	Thus, we are left with the computation of the $\Maslov$-part which has been calculated in~\cite[$(6.16)$]{Duistermaat_Inventmath_1975} 
	(see also~\cite[Sec. 3]{Meinrenken_JGP_1994}). 
	For completeness, we briefly outline the main steps.  
	First, one constructs the Keller-Maslov bundle $\bbL \to \cC_{T}$ from $\bbL_{C}, \bbL_{\varGamma}$ and $\bbL_{\varLambda}$ following the procedure detailed in Appendix~\ref{sec: composition_halfdensity_Maslov_canonical_relation}. 
	Let $\varphi$ resp. $\phi$ be generating functions of $C_{\uppi_{*} \varDelta^{*}}$ resp. $\cC_{T}$. 
	Then $\psi := \varphi + \phi$ locally generates $\Lambda_{T}$. 
	Suppose that we partition $[0, 1]$ as $0 = s_{0} < s_{1} < \ldots < s_{N} = 1$ such that $\{ \mathcal{L}_{\alpha} \}_{k=0, \ldots, N}$ is a conic covering of $\Lambda_{T}$ around $\gamma (s_{\alpha})$ and that $\{ \psi_{\alpha} \}$ are the corresponding generating functions. 
	Then, utilising $\psi$ in the formula, for example, 
	in~\cite[$(5)$]{Meinrenken_ReptMathPhys_1992} 
	one obtains the Maslov index. 
\end{proof}
%
%
%
%
%
%
%
%
%
%
\section{Proof of Weyl law}  
As shown in Theorem~\ref{thm: trace_formula_t_zero} and Theorem~\ref{thm: trace_formula_t_T} that $\Tr U_{t}$ has singularities at $t=0, T$. 
Our aim is to only cut the $t = 0$ singularity. 
In order to do so, one introduces a Schwartz function $\chi (\lambda)$ on $\R$ such that 
(see e.g.~\cite[Thm. 7.5]{Wunsch_Zuerich_2008}) 
(i) $\chi (\lambda) > 0$ for all $\lambda$, 
(ii) $(\cF \chi) (0) \equiv 1$, 
(iii) $(\cF \chi) (t) \equiv (\cF \chi) (-t)$, 
(iv) $\supp (\cF_{\lambda \mapsto t} \chi) \subset (- \varepsilon, \varepsilon)$, where $\varepsilon > 0$ is sufficiently small. 
Then employing expression of $\Tr U_{0}$ from Theorem~\ref{thm: trace_formula_t_zero} and the aforementioned properties of $\chi$, we obtain 
\begin{eqnarray}
	\cF_{t \mapsto \lambda}^{-1} \big( \cF_{\cdot \mapsto t} (\chi) \Tr U_{t} \big) 
	& \approx & 
	\int \rd \tau \, \delta (\tau - \lambda) r \frac{\res (H^{1-d})}{(2 \pi)^{d-1}} \tau^{d-2} + \ldots 
	\nonumber \\ 
	& = & 
	r (d-1) \frac{\vol (\cN_{H \leq 1})}{(2 \pi)^{d-1}} \lambda^{d-2} + \ldots, 
\end{eqnarray}
where $f \approx g$ means that $\nicefrac{f}{g} \to 1$ as $t \to 0$  for any functions $f (t), g(t)$. 
That means that 
\begin{equation}
	(\chi * \rd \fN) (\lambda) = r (d-1) \frac{\vol (\cN_{H \leq 1})}{(2 \pi)^{d-1}} \lambda^{d-2} + O (\lambda^{d-3}), 
\end{equation}
where $*$ denotes the convolution. 
Employing 
\begin{equation}
	(\chi * \fN) (\lambda) 
	= 
	\int_{- \infty}^{\lambda} (\chi * \rd \fN) (\mu) \, \rd \mu 
	=
	r \frac{\vol (\cN_{H \leq 1})}{(2 \pi)^{d-1}} \lambda^{d-1} + O (\lambda^{d-2}). 
\end{equation}
The Weyl law entails from 
\begin{equation}
	\fN (\lambda) = (\fN * \chi) (\lambda) - \int \big( \fN (\lambda - \mu) - \fN (\lambda) \big) \, \chi (\mu) \, \rd \mu
\end{equation}
together with the facts $\fN (\lambda - \mu) - \fN (\lambda) \lessapprox \langle \mu \rangle^{d-1} \langle \lambda \rangle^{d-2}$ and $\int \chi (\mu) \, \rd \mu = 1$. 
%
%
%
%
%
%
%
%
%
%
\appendix 
\section{Composition of bundle valued densities on canonical relations}
\label{sec: density_canonical_relation}
In this appendix we accumulate some well-known facts about half-densities on a canonical relation and their composition, as used extensively in the principal symbol computation. 
Recall, the punctured cotangent bundle $\dotCoTanM$ over a manifold $M$ is a conic symplectic manifold, as it naturally carries the canonical symplectic form $\upsigma_{\ms M}$ which is given by $\upsigma_{\ms M} = \rd x^{i} \wedge \rd \xi_{i}$ in any homogeneous symplectic coordinates $(x^{i}, \xi_{i})$ for $\dotCoTanM$. 
If $N$ is a manifold, not necessarily of the same dimension as $M$, then $\big( \dotCoTanM \times \dotCoTanN, \upsigma^{\pm} := \pr_{\ms M}^{*} \upsigma_{\ms M} \pm \pr_{\ms N}^{*} \upsigma_{\ms N} \big)$ are conic symplectic manifolds, where $\pr_{\ms M}, \pr_{\ms N} : M \times N \to M, N$ are the Cartesian projectors. 
Throughout the article, only the symplectic form $\upsigma^{-}$ has been used, i.e., by a  canonical relation $C \subset \dotCoTanM \times \dotCoTanN$ from $N$ to $M$, it is meant that $C$ is a Lagrangian submanifold of $\dotCoTanM \times \dotCoTanN$ with respect to $\upsigma^{-}$, or equivalently, $C'$ is a canonical relation with respect to $\upsigma^{+}$. 
If $C$ is conic then it is called a homogeneous canonical relation~\cite[Def. 4.1.2]{Hoermander_ActaMath_1971}  
(see also~\cite[Def. 21.2.12]{Hoermander_Springer_2007}). 
%
%
%
%
%
%
%
%
%
%
\subsection{Half-density bundle over a canonical relation} 
\label{sec: volume_canonical_relation}
Any homogeneous canonical relation $C \subset \dotCoTanMN \cong \dotCoTanM \times \dotCoTanN$ can always be locally parametrised by a 
non-degenerate phase function $\varphi (x, y; \theta)$, defined locally on $M \times N \times \dot{\R}^{n}, n := \dim M + \dim N$~\cite[Thm. 3.1.3]{Hoermander_ActaMath_1971} 
(see also, e.g.~\cite[Thm. 21.2.16, 21.2.18]{Hoermander_Springer_2007},~\cite[pp. 418-419]{Treves_Plenum_1980}). 
Let 
\begin{equation}
	\sC := (\grad_{\theta} \varphi)^{-1} (0)  
\end{equation}
be the fibre-critical manifold of $\varphi$.  
Then $C$ is locally obtained by the homogeneous immersion~\cite[p. 134]{Hoermander_ActaMath_1971} 
(see also, e.g.~\cite{Meinrenken_ReptMathPhys_1992}) 
\begin{equation}
	\jmath : \sC \to C, ~ \xytheta \mapsto 
	\jmath \xytheta := (x, \rd_{x} \varphi; y, \rd_{y} \varphi). 
\end{equation}
Since the map~\cite[p. 143, Prop. 4.1.3]{Hoermander_ActaMath_1971} 
(for details, see e.g.~\cite[$(4.1.1)$]{Duistermaat_Birkhaeuser_2011},~\cite[p. 440]{Treves_Plenum_1980})  
\begin{equation} \label{eq: surjective_map_M_N_Rn_Rn}
	M \times N \times \dot{\R}^{n} \ni \xytheta \mapsto \bigg( \parDeri{\theta_{1}}{\varphi} \xytheta, \ldots, \parDeri{\theta_{n}}{\varphi} \xytheta \bigg) \in \dot{\R}^{n}
\end{equation} 
is surjective, $\sC$ is endowed with the measure $\dVol_{\ms \sC} := \delta (\grad_{\theta} \varphi)$ defined by taking the quotient of the measure on $M \times N \times \dot{\R}^{n}$ by the pullback of the delta-distribution under the mapping~\eqref{eq: surjective_map_M_N_Rn_Rn}. 
Thereby,   
\begin{equation}
	\dVol_{\ms C} := \big( (\rd \jmath)^{-1} \big)^{*} \dVol_{\ms \sC}
\end{equation}
gives a natural measure on $C$ and hence one has the bundle $\varOmega^{\nicefrac{1}{2}} C \to C$ of half-densities over $C$ whose local sections are $\sqrt{|\dVol_{\ms C}|}$.  

The situation becomes much simpler when the canonical relation $\varGamma$ is the graph of a symplectomorphism from $\dotCoTanN$ to $\dotCoTanM$. 
In this case $\varGamma$ is a symplectic manifold with respect to the symplectic form 
$\upsigma_{\ms \varGamma} := \pr_{\ms M}^{*} \upsigma_{\ms M} = \pr_{\ms N}^{*} \upsigma_{\ms N}$
(see e.g.~\cite[p. 25]{Hoermander_Springer_2009}) 
and we have the corresponding Liouville volume form 
\begin{equation}
	\rd \mathsf{v}_{\ms \varGamma} := \frac{(-1)^{d (d-1)/2}}{d!} \upsigma_{\ms \varGamma}^{d}. 
\end{equation}
%
%
%
%
%
%
%
%
%
%
\subsection{Composition of half-densities} 
\label{sec: composition_halfdensity_canonical_relation}
Let $C \subset \dotCoTan (M \times O)$ and $\varLambda \subset \dotCoTan (O \times N)$ be closed conic canonical relations, where $O$ is a manifold whose dimension is not necessarily equal to that of either $M$ or $N$. 
In order to guarantee that the composition $A \circ B$ of two Fourier integral operators $A$ resp. $B$ associated with $C$ resp. $\varLambda$ is a Fourier integral operator, it is sufficient 
(see e.g.~\cite[Thm. 25.2.3]{Hoermander_Springer_2009};  
an exhaustive list of original references is available in~\cite[Sec. 2.1.1]{Islam})  
to assume that $C \circ \varLambda$ is clean, 
proper\footnote{This ensures that $C \circ \varLambda$ is closed in $\dotCoTanMN$.} 
and 
connected\footnote{This implies no self-intersection} 
(for details of these terminologies, see, for instance, the 
monographs~\cite[App. C.3, Thm. 21.2.14, Prop. 21.2.19]{Hoermander_Springer_2007},~\cite[pp. 457-458]{Treves_Plenum_1980}). 
Our aim is to define a half-density on $C \circ \varLambda$, given half-densities on $C$ and $\varLambda$. 
To do so, it is convenient to introduce the notation 
\begin{equation} \label{eq: def_C_star_varLambda}
	C {\scriptsize \text{\FiveStarOpen}} \varLambda := (C \times \varLambda) \bigcap \big( \dotCoTanM \times \varDelta (\dotCoTan O) \times \dotCoTanN \big) 
\end{equation}
and the Cartesian projector $\varPi: \dotCoTanM \times \varDelta (\dotCoTan O) \times \dotCoTanN \to \dotCoTanM \times \dotCoTanN$, so that the clean composition between $C$ and $\varLambda$ is given by 
\begin{subequations}
	\begin{eqnarray}
		C \circ \varLambda 
		& = &   
		\varPi (C {\scriptsize \text{\FiveStarOpen}} \varLambda),  
		\label{eq: def_composition_canonical_relation}
		\\ 
		& = & 
		\big\{ \xxiyeta \in \dotCoTanM \times \dotCoTanN \,|\, \exists (z, \zeta) \in \dotCoTan O : (x, \xi; z, \zeta) \in C, (z, \zeta; y, \eta) \in \varLambda \big\}.  
		\nonumber \\ 
		\label{eq: def_pointwise_composition_canonical_relation}  
	\end{eqnarray}
\end{subequations} 
The excess $e$ of $C \circ \varLambda$ is 
(see e.g.~\cite[Thm. 4.2.2]{Guillemin_InternationalP_2013}) 
\begin{equation} \label{eq: excess_fibre_projection_clean_composition}
	e = \dim \kF, 
\end{equation}
where 
\begin{eqnarray} \label{eq: def_fibre_projection_clean_composition}
	\kF & := & \{ \kF_{\xxiyeta} \,|\, \xxiyeta \in C \circ \varLambda \}, 
	\nonumber \\ 
	\kF_{\xxiyeta} & := & \{ (x, \xi; z, \zeta; z, \zeta; y, \eta) \in C {\scriptsize \text{\FiveStarOpen}} \varLambda \,|\, \varPi (x, \xi; z, \zeta; z, \zeta; y, \eta) = \xxiyeta \in C \circ \varLambda \} 
	\nonumber \\ 
\end{eqnarray}
is the set of compact connected fibres $\kF_{\xxiyeta}$ of the smooth fibration
\begin{equation} \label{eq: varPi}
	\varPi : C {\scriptsize \text{\FiveStarOpen}} \varLambda \to C \circ_{e} \varLambda. 
\end{equation}
In terms of generating functions, if $\varphi (x, z; \theta)$ and $\phi (z, y; \vartheta)$ are non-degenerate phase functions for $C$ and $\varLambda$ in open conic neighbourhoods of $(x_{0}, z_{0}; \theta^{0}) \in M \times O \times \dot{\R}^{\nM}$ and $(z_{0}, y_{0}; \vartheta^{0}) \in O \times N \times \dot{\R}^{\nN}$ such that $(x_{0}, z_{0}; \theta^{0}) \in \sC_{\varphi}$ and $(z_{0}, y_{0}; \vartheta^{0}) \in \sC_{\phi}$, respectively, with $\rd_{z} \varphi \, (x_{0}, z_{0}; \theta^{0}) + \rd_{z} \phi \, (z_{0}, y_{0}; \vartheta^{0}) = 0$, then 
$\psi (x, z, y; \theta, \vartheta) := \varphi (x, z; \theta) + \phi (z, y; \vartheta)$
is a clean phase function with excess $e$ parametrising $C \circ_{e} \varLambda$. 
In addition, if $\psi$ is non-degenerate in $(z'; \theta', \vartheta')$ then the $e$ variables $(z''; \theta'', \vartheta'')$ parametrise $\kF$
(see e.g.~\cite[Prop. 21.2.19]{Hoermander_Springer_2007}). 

In order to achieve the desired composition formula, one employs the isomorphism~\cite[Lem. 5.2]{Duistermaat_Inventmath_1975} 
(see also, e.g.~\cite[Thm. 7.1.1, $(7.4)$]{Guillemin_InternationalP_2013},~\cite[Thm. 21.6.7]{Hoermander_Springer_2007})  
\begin{equation} \label{eq: isom_tensor_halfdensity_clean_composition}
	\halfDen_{(x, \xi; z, \zeta)} C \otimes \halfDen_{(z, \zeta; y, \eta)} \varLambda \cong \halfDen_{(x, \xi; z, \zeta; y, \eta)} \kF \otimes \halfDen_{\xxiyeta} (C \circ \varLambda),  
\end{equation} 
where $\kF$ is the set of fibres $\kF_{\xxiyeta}$ of the fibration $\varPi$. 
This entails that by taking tensor product between half-densities on $C$ and $\varLambda$ followed by intersecting with the diagonal gives a half-density on $\kF$ times a half-density on $C \circ \varLambda$, and therefore, integrating over the fibre $\kF_{\xxiyeta}$ we achieve the desired  composition of half-densities on canonical relations~\cite[p. 64]{Duistermaat_Inventmath_1975}:    
\begin{eqnarray} \label{eq: def_composition_halfdensity_canonical_relation}
	\cdot \diamond \cdot &:& 
	C^{\infty} \big( C; \halfDen C \big) \times C^{\infty} \big( \varLambda; \halfDen \varLambda \big) \to C^{\infty} \big( C \circ \varLambda; \halfDen (C \circ \varLambda) \big), 
	\nonumber\\ 
	&& 
	(\mu, \nu) \mapsto \mu \diamond \nu \, \xxiyeta := \int_{\kF_{\xxiyeta}} (\mu \otimes \nu) \big( \varPi^{-1} \xxiyeta \big). 
\end{eqnarray}
If $e = 0$ then the isomorphism~\eqref{eq: isom_tensor_halfdensity_clean_composition} reduces to 
$\halfDen C \otimes \halfDen \varLambda \cong \halfDen (C \circ \varLambda)$~\cite[p. 179]{Hoermander_ActaMath_1971}  
(see also, e.g.~\cite[$(7.3)$]{Guillemin_InternationalP_2013}) 
and hence~\eqref{eq: def_composition_halfdensity_canonical_relation} simply becomes the pointwise evaluation 
(see e.g.~\cite[Thm. 4.2.2]{Duistermaat_Birkhaeuser_2011}) 
\begin{equation}
	\mu \diamond \nu \, \xxiyeta = \sum_{(z, \zeta) \,|\, (x, \xi; z, \zeta) \in C, (z, \zeta; y, \eta) \in \varLambda} \mu (x, \xi; z, \zeta) \, \nu (z, \zeta; y, \eta).  
\end{equation} 
We call the pairs $\xxi, \yeta$ (resp. $(z, \zeta)$) in the preceding equation as the output (rep. input) variables of the composition following the terminology from~\cite[Lem. 8.4]{Strohmaier_AdvMath_2021}. 
%
%
%
%
%
%
%
%
%
%
\subsection{Composition of Keller-Maslov bundle valued halfdensities} 
\label{sec: composition_halfdensity_Maslov_canonical_relation}
The principal symbol of a scalar Lagrangian distribution associated with the canonical relation $C$ is the Keller-Maslov bundle $\bbL \to C$-valued density on $C$
(see e.g.~\cite[Sec. 25.1]{Hoermander_Springer_2009} 
and the original references cited therein).  
Hence, we are compelled to extend the composition~\eqref{eq: def_composition_halfdensity_canonical_relation} on sections of $\bbL \otimes \varOmega^{1/2} C$. 
The first step, thus is to construct the Keller-Maslov bundle $\Maslov \to C \circ \varLambda$ given $\bbL \to C$ and $\T \to \varLambda$. 
This is achieved by employing the phase function $\varphi, \phi, \psi$ for $\bbL, \T, \Maslov$, respectively~\cite[Sec. 3.3]{Hoermander_ActaMath_1971}; 
for details, see e.g. the 
monographs~\cite[pp. 408-412]{Treves_Plenum_1980},~\cite[pp. 132-138]{Guillemin_InternationalP_2013}.   
In particular, $\bbL$ is an associated complex-line bundle with structure group $\Z / 4 \Z$ that is trivial as a vector bundle. 
Its defining sections are $\{ l_{i} := \exp \big( \ri \pi \mathrm{Hess} (\varphi_{i}) / 4 \big) \}$ with $l_{i} = c_{ij} l_{j}$ for some constant $c_{ij} \in \C$ such that $|c_{ij}| = 1$, where $\mathrm{Hess} \, \varphi$ denotes the Hessian of $\varphi$ 
(see e.g.~\cite[Sec. 5.13]{Guillemin_InternationalP_2013} for details).   

Next, one employs the isomorphism~\cite[$(5.7)$]{Duistermaat_Inventmath_1975} 
(see also, e.g.~\cite[$(5.29)$]{Guillemin_InternationalP_2013})   
\begin{equation} \label{eq: isom_tensor_Maslov_bundle}
	\bbL \boxtimes \T := \varPi_{\ms C}^{*} \bbL \otimes \varPi_{\ms \varLambda}^{*} \T \cong \varPi^{*} \Maslov 
\end{equation}
where $\varPi_{\ms C}, \varPi_{\ms \varLambda} : C {\scriptsize \text{\FiveStarOpen}} \varLambda \to C, \varLambda$ are the projections, in order to make the identification  
\begin{equation}
	(\varPi^{*} \m) (x, \xi; z, \zeta; z, \zeta; y, \eta) = (\varPi_{\ms C}^{*} \bbl \otimes \varPi_{\ms \varLambda}^{*} \bbt) (x, \xi; z, \zeta; z, \zeta; y, \eta) 
\end{equation}
for any given sections $\bbl$ resp. $\bbt$ of $\bbL$ resp. $\T$ with a section $\m$ of $\M$. 
The desired extension of~\eqref{eq: def_composition_halfdensity_canonical_relation} is then  
\begin{eqnarray} \label{eq: def_composition_halfdensity_Maslov_canonical_relation}
	&& 
	C^{\infty} \big( C; \bbL \otimes \varOmega^{\frac{1}{2}} C \big) 
	\times 
	C^{\infty} \big( \varLambda; \T \otimes \varOmega^{\frac{1}{2}} \varLambda \big) 
	\to 
	C^{\infty} \big( C \circ \varLambda; \Maslov \otimes \varOmega^{\frac{1}{2}} (C \circ \varLambda) \big), 
	\nonumber \\ 
	&& 
	(\bbl \otimes \mu, \bbt \otimes \nu) \mapsto \m \otimes (\mu \diamond \nu).  
\end{eqnarray}
%
%
%
%
%
%
%
%
%
%
\subsection{Composition of vector bundle valued half-densities} 
\label{sec: composition_halfdensity_Maslov_vector_bundle_canonical_relation}
Finally, we consider any vector bundles $\sE \to M, G \to O, F \to N$ and so $\Hom{G, \sE} \to M \times O$ is a homomorphism bundle, forbye, $\widetilde{\mathrm{Hom}} (G, \sE) \to C$ represents the pullback of $\Hom{G, \sE}$ to $\coTan (M \times O)$ followed by restriction to $C$. 
One then extends~\eqref{eq: def_composition_halfdensity_Maslov_canonical_relation} by tensoring with $\widetilde{\mathrm{Hom}} (G, \sE), \widetilde{\mathrm{Hom}} (F, G), \widetilde{\mathrm{Hom}} (F, \sE)$, respectively to achieve our final composition rule~\cite[$(25.2.10)$]{Hoermander_Springer_2009}  
\begin{eqnarray} \label{eq: def_composition_halfdensity_Maslov_vector_bundle_canonical_relation}
	&& 
	C^{\infty} \big( C; \bbL \otimes \halfDen C \otimes \widetilde{\mathrm{Hom}} (G, \sE) \big)
	\times 
	C^{\infty} \big( \varLambda; \T \otimes \halfDen \varLambda \otimes \widetilde{\mathrm{Hom}} (F, G) \big) 
	\nonumber\\ 
	&& \hspace*{6.0cm} \to C^{\infty} \big( C \circ \varLambda; \Maslov \otimes \halfDen (C \circ \varLambda) \otimes \widetilde{\mathrm{Hom}} (F, \sE) \big), 
	\nonumber\\ 
	&& 
	(\bbl \otimes a, \bbt \otimes b) \mapsto \m \otimes (a \diamond b), 
	\nonumber \\ 
	&& 
	(a \diamond b) \xxiyeta := \int_{\kF_{\xxiyeta}} (a \otimes b) \big( \varPi^{-1} \xxiyeta \big).  
\end{eqnarray}
As before, the hindmost expression reduces to standard composition of homomorphisms when $e = 0$: 
\begin{equation}
	(a \diamond b) \xxiyeta = \sum_{(z, \zeta) \,|\, (x, \xi; z, \zeta) \in C, (z, \zeta; y, \eta) \in \varLambda} a (x, \xi; z, \zeta) \big( b (z, \zeta; y, \eta) \big). 
\end{equation}
%
%
%
%
%
%
%
%
%
%
\section{Restriction map on a vector bundle}
\label{sec: restriction_op}
Let $\sE \to M$ be a vector bundle over a manifold $M$ and $\iota_{\ms \varSigma} : \varSigma \hookrightarrow M$ an immersed submanifold of codimension $\codim \varSigma := \dim M - \dim \varSigma$.  
Given a fibrewise isomophism $\hat{\iota}_{\ms \varSigma} : \sE_{\iota_{\ms \varSigma} (x')} \cong (\iota_{\ms \varSigma}^{*} \sE)_{x'}$, the restriction map  
\begin{equation}
	\iota_{\ms \varSigma}^{*} : \comSecE \to C^{\infty} (\varSigma; \sE_{\varSigma})
\end{equation}
is defined by the pullback of the morphism $(\iota_{\ms \varSigma}, \hat{\iota}_{\ms \varSigma})$. 
This result is well-known for 
scalar case~\cite[Sec. 5.1]{Duistermaat_Birkhaeuser_2011}  
and translates in a straightforward way for vector bundles yet briefly presented here for the sake of completeness. 

It is well-known that $\varSigma$ admits an adapted atlas, i.e., for each $x' \in \varSigma$, there exist charts $(V, \rho)$ resp. $(U, \kappa)$ for $\varSigma$ at $x'$ resp. for $M$ at $\iota_{\ms \varSigma} (x')$ such that $\iota_{\ms \varSigma} (V) \subset U$ and $\kappa \circ \iota_{\ms \varSigma} \circ \rho^{-1} (\boldsymbol{x}') = (0, \ldots, 0, x^{\codim \varSigma +1}, \ldots, x^{d})$ is a chart for $\varSigma$. 
After trivialising $\sE$, it follows that the Schwartz kernel $\fR_{\varSigma}$ of $\iota_{\ms \varSigma}^{*}$ is given by 
\begin{eqnarray}
	\fR_{\varSigma} (x', y) & := & \int_{\Rd} \re^{\ri \varphi} \frac{\rd \eta}{(2 \pi)^{d}}, 
	\nonumber \\ 
	\varphi & := & (x^{i} - y^{i}) \eta_{i} - y^{j} \eta_{j}, i = \codim \varSigma + 1, \ldots, d; j = 1, \ldots, \codim \varSigma. 
\end{eqnarray}
The phase function $\varphi$ is linear and non-degenerate whose fibre-critical set is $\sC = \{ (x', y; \eta) \in \varSigma \times M \times \dot{\R}^{d} | \iota (x') = y \}$. 
Thus, $\fR_{\varSigma}$ is a Lagrangian distribution of order $\codim \varSigma/4$ associated with the canonical relation $\varLambda_{\varSigma}' := \big\{ (x', \xi'; x, \xi) \in \coTanM_{\varSigma} \times \dotCoTanM \,\big|\, x = \iota (x'), \xi' = \xi|_{\mathrm{T}_{x'} M_{\varSigma}} \big\}$ 
(see~\cite[(2.4.4), (2.4.22), (5.1.2)]{Duistermaat_Birkhaeuser_2011} for scalar version).   

One discerns that $\varLambda_{\varSigma} \subset \coTanM_{\varSigma} \times \dotCoTanM$ is not a homogeneous canonical relation because it is not guaranteed that restriction of all elements in $\dotCoTanM$ to $\varSigma$ is non-zero. 
Geometrically speaking, the complication arises due to the elements of the conormal bundle $\varSigma^{\perp *} \subset \coTanM$. 
To circumvent this difficulty, one temporarily introduces a cutoff function $\chi$ on $\varSigma^{\perp *}$ and sets~\cite[Sec. 5.2]{Toth_GFA_2013}   
\begin{equation}
	\fR_{\chi} (x', y) := \int_{\Rd} \re^{\ri \varphi} \big( 1 - \chi \yeta \big) \frac{\rd \eta}{(2 \pi)^{d}}
\end{equation}   
so that such elements do not occur in the support of $1 - \chi$. 
Then the respective operator $\iota_{\chi}^{*}$ is a homogeneous Fourier integral operator of order $\codim \varSigma/4$ associated with the homogeneous canonical relation $\varLambda_{\chi} \subset \dotCoTanM_{\varSigma} \times \dotCoTanM$. 
Putting altogether: 
\begin{subequations} \label{eq: restriction_kernel}
	\begin{eqnarray}
		&& 
		\fR_{\chi} \in I^{\codim \varSigma / 4} \big( \varSigma \times M, \varLambda_{\chi}'; \Hom{\sE, \sE_{\varSigma}} \big), 
		\label{eq: restriction_kernel_Lagrangian_dist}
		\\ 
		&& 
		\varLambda_{\chi}' := \big\{ (x', \xi'; x, \xi) \in \dotCoTanM_{\varSigma} \times \dotCoTanM \,\big|\, x = \iota (x'), \xi' = \xi|_{\mathrm{T}_{x'} M_{\varSigma}} \big\},  
		\label{eq: def_canonical_relation_restriction} 
		\\ 
		&& 
		\symb{\fR_{\chi}} |_{\supp (1 - \chi)} (x', \xi'; x, -\xi) = (2 \pi)^{-\frac{\codim \varSigma}{4}} \one_{ \Hom{\sE, \sE_{\varSigma}} } |\dVol_{\ms \varLambda} (x', \xi'; x, -\xi)|^{\frac{1}{2}} \otimes \mathbbm{l},   
		\label{eq: symbol_restriction}
	\end{eqnarray}
\end{subequations} 
where $|\dVol_{\ms \varLambda}|$ is the natural $1$-density on $\varLambda_{\chi}$ and $\mathbbm{l}$ is a section of the Keller-Maslov bundle $\bbL_{\varLambda} \to \varLambda$. 
The canonical symplectic form $\sigma$ on $\coTanM$ is given by $\sigma = \rd x^{i} \wedge \rd \xi_{i} + \rd x^{j} \wedge \rd \xi_{j}$ in the adapted coordinates and so $\dVol_{\ms \varLambda} = \rd x' \wedge \rd \xi$ is the induced density on $\varLambda_{\varSigma}$. 
The bundle $\bbL_{\varLambda}$ comprises global constant sections constructed from $\varphi$.  

Equivalently, one can say that $\iota_{\ms \varSigma}^{*}$ is not a homogeneous Fourier integral operator due to plausible occurrence of elements in $\varSigma^{\perp *}$ in its canonical relation and the cutoff above can be emulated by composing with a $P \in \PsiDO{0}{M; \sE}$ with $\WF{P} \cap \dot{\varSigma}^{\perp *} = \emptyset$. 
The principal symbol of $\fR_{\varSigma}$ is then $(2 \pi)^{- \codim \varSigma / 4} \one |\dVol_{\ms \varLambda}|^{1/2} \otimes \mathbbm{l}$ over $\WF{P}$ 
(see e.g.~\cite[Lem. 8.3]{Strohmaier_AdvMath_2021} for the scalar version). 

Any distribution on $\sE$ can be restricted to $\sE_{\varSigma}$ when its wavefront set is disjoint with $\varSigma^{\perp *}$. 
In particular, if $\mathcal{L}$ is such that the wavefront set of every $u \in I^{m} (M, \mathcal{L}; \sE)$ is disjoint from the conormal bundle $\varSigma^{\perp *}$ of $\varSigma \subset \sM$ then the restriction operator can be extended to a sequentially continuous linear operator 
\begin{eqnarray}
	\iota_{\chi}^{*} : I^{m} (M, \mathcal{L}; \sE) \to I^{m + \frac{\codim \varSigma}{4}} (\varSigma, \mathcal{L} |_{\tangent M_{\varSigma}}; \sE_{\varSigma}), 
	\quad 
	\symb{\iota_{\chi}^{*} u} (x', \xi') \equiv (2 \pi)^{- \frac{\codim \varSigma}{4}} \symb{u} (x', \xi').   
\end{eqnarray}
%
%
%
%
%
%
%
%
%
%


\begin{bibdiv}
\begin{biblist}
\bib{Abraham_AMS_1978}{book}{
	title=		{Foundations of Mechanics},
	author=		{Abraham, R.},
	author=		{Marsden, J. E.},
	volume=		{364},
	edition=	{2nd},
	year=		{1978; AMS reprint 2008},
	address=	{USA}, 
	publisher=	{AMS Chelsea Publishing}
}


\bib{Avetisyan_JST_2016}{article}{
	title=         {Spectral asymptotics for first order systems},
	volume=        {6},
	number=        {4},
	journal=       {J. Spectr. Theory},
	publisher=     {European Mathematical Society Publishing House},
	author=        {Avetisyan, Z.},
	author=		   {Fang, Y.-L.},
	author=        {Vassiliev, D.},
	year=          {2016},
	pages=         {695 -- 715}, 
	archivePrefix= {arXiv},
	eprint=		   {arXiv:1512.06281[math.SP]},
	primaryClass = {math.SP}
}

\bib{Baer_EMS_2007}{book}{
	author=			{B\"{a}r, C.},
	author= 		{Ginoux, N.},
	author= 		{Pf\"{a}ffle, F.},
	title=			{Wave Equations on {L}orentzian Manifolds and Quantization},
	series=			{ESI Lectures in Mathematics and Physics},
	publisher=		{European Mathematical Society},
	address=		{Germany},
	year=			{2007},
	eprint= 		{arXiv:0806.1036v1[math.DG]}
}

\bib{Baer_Springer_2012}{article}{
 	author=			{B{\"a}r, C.},
 	author= 		{Ginoux, N.},
 	title=			{Classical and Quantum Fields on {L}orentzian Manifolds},
 	book={
 		editor=		{B{\"a}r, C.},
 		editor=		{Lohkamp, J.},
 		editor=		{Schwarz, M.},
 		title=		{Global Differential Geometry},
 		year=		{2012},
 		publisher=	{Springer},
 		address=	{Berlin, Heidelberg},
 	},
 	pages=			{359 -- 400},
 	isbn=			{978-3-642-22842-1},
 	archivePrefix= 	{arXiv},
 	eprint       = 	{arXiv:1104.1158[math-ph]},
 	primaryClass = 	{math-ph}
}

\bib{Bartolo_NonlinearAnal_2001}{article}{
	title = {Periodic trajectories on stationary Lorentzian manifolds},
	journal = {Nonlinear Anal. Theory Methods Appl.},
	volume = {43},
	number = {7},
	pages = {883-903},
	year = {2001},
	issn = {0362-546X},
	author = {Bartolo, R.}
}

\bib{Baum_AGAG_1996}{article}{
	author=		{Baum, H.},
	author=		{Kath, I.},
	title=		{Normally hyperbolic operators, the {H}uygens property and conformal geometry},
	journal=	{Ann. Glob. Anal. Geom.},
	volume=		{14},
	date=		{1996},
	number=		{4},
	pages=		{315 -- 371}
}

\bib{Bolte_FoundPhys_2001}{article}{
	author=	{Bolte, J.},
	title=		{Semiclassical Expectation Values for Relativistic Particles with Spin 1/2},
	journal=	{Found. Phys.},
	year=		{2001},
	month=		{Feb},
	day=		{01},
	volume=	{31},
	number=	{2},
	pages=		{423 -- 444},
	archivePrefix=	{arXiv},
	eprint= 	{arXiv:0009052[nlin.CD]},
	primaryClass= 	{nlin.CD}
}

\bib{Bolte_JPA_2004}{article}{
	year=		{2004},
	month=		{jun},
	publisher=	{{IOP} Publishing},
	volume=	{37},
	number=	{24},
	pages=		{6359 -- 6373},
	author=	{Bolte, J.},
	author={Glaser, R.},
	title=		{Zitterbewegung and semiclassical observables for the {D}irac equation},
	journal=	{J. Phys. A: Math. Gen.},
	archivePrefix=	{arXiv},
	eprint=	{arXiv:0402154[quant-ph]},
	primaryClass=	{quant-ph}
}

\bib{Bolte_PRL_1998}{article}{
	title=		{Semiclassical Time Evolution and Trace Formula for Relativistic Spin-1/2 Particles},
	author=		{Bolte, J.},
	author=		{Keppeler, S.},
	journal=	{Phys. Rev. Lett.},
	volume=		{81},
	issue=		{10},
	pages=		{1987 -- 1991},
	numpages=	{0},
	year=		{1998},
	month=		{Sep},
	publisher=	{American Physical Society},
	archivePrefix= {arXiv},
	eprint       = {arXiv:9805041[quant-ph]},
	primaryClass = {quant-ph}
}

\bib{Bolte_AnnPhys_1999}{article}{
	title=		{A Semiclassical Approach to the {D}irac Equation},
	journal=	{Ann. Phys.},
	volume=	{274},
	number=	{1},
	pages=		{125 -- 162},
	year=		{1999},
	issn=		{0003-4916},
	author=	{Bolte, J.},
	author={Keppeler, S.},
	archivePrefix= {arXiv},
	eprint       = {arXiv:9811025[quant-ph]},
	primaryClass = {quant-ph}
}

\bib{Branson_JFA_1992}{article}{
	title=	{Residues of the eta function for an operator of Dirac type},
	journal={J. Funct. Anal.},
	volume=	{108},
	number=	{1},
	pages=	{47 -- 87},
	year=	{1992},
	issn=	{0022-1236},
	author=	{Branson, T. P.},
	author=	{Gilkey, P. B.}
}

\bib{Candela_AdvMath_2008}{article}{
	title=	{Global hyperbolicity and {P}alais–{S}male condition for action functionals in stationary spacetimes},
	journal=	{Adv. Math.},
	volume=	{218},
	number=	{2},
	pages=	{515--536},
	year=		{2008},
	issn=		{0001-8708},
	author=	{Candela, A.M.}, 
	author={Flores, J.L.},
	author={S\'{a}nchez, M.},
	eprint       ={arXiv:0610175[math.DG]}
}


\bib{Capoferri}{article}{ 
	author={Capoferri, M.},
	author={Murro, S.},
	title={Global and microlocal aspects of {D}irac operators: propagators and {H}adamard states},
	eprint={arXiv:2201.12104[math.AP]},
	year={2022}
}


\bib{Capoferri_2020}{article}{ 
	title={Global propagator for the massless {D}irac operator and spectral asymptotics},
	author={Capoferri, M.},
	author={Vassiliev, D.},
	journal={arXiv:2004.06351[math.AP]},
	year={2020}
}

\bib{Chazarain_InventMath_1974}{article}{
	author={Chazarain, J.},
	title={Formule de {P}oisson pour les vari{\'e}t{\'e}s riemanniennes},
	journal={Invent. math},
	year={1974},
	month={Mar},
	day={01},
	volume={24},
	number={1},
	pages={65--82},
	issn={1432-1297}
}

\bib{Chazarain_CPDE_1980}{article}{
	author=		{Chazarain, J.},
	title=		{Spectre d'Un hamiltonien quantique et mecanique classique},
	journal=	{Commun. Partial Differ. Equ.},
	volume=		{5},
	number=		{6},
	pages=		{595 -- 644},
	year =		{1980},
	publisher=	{Taylor \& Francis}
}

\bib{Chervova_JST_2013}{article}{
	title=      {The spectral function of a first order elliptic system},
	volume=     {3},
	number=     {3},
	journal=    {J. Spectr. Theory},
	publisher=  {European Mathematical Society Publishing House},
	author=     {Chervova, O.},
	author=		{Downes, R.},
	author=		{Vassiliev, D.},
	year=       {2013},
	pages=      {317 -- 360}, 
	archivePrefix={arXiv},
	eprint=		{arXiv:1208.6015[math.SP]},
	primaryClass={math.SP}
}


\bib{Duistermaat_Birkhaeuser_2011}{book}{
	title=		{Fourier Integral Operators},
	author=		{Duistermaat, J. J.},
	series=		{Modern Birkh\"{a}user Classics},
	url=		{https://www.springer.com/gb/book/9780817681074},
	year=		{2011},
	address=	{New York}, 
	publisher=	{Birkh\"{a}user}
}

\bib{Duistermaat_Inventmath_1975}{article}{
	author=		{Duistermaat, J. J.},
	author=		{Guillemin, V. W.},
	title=		{The spectrum of positive elliptic operators and periodic bicharacteristics},
	journal=	{Invent. math.},
	year=		{1975},
	month=		{Feb},
	day=		{01},
	volume=		{29},
	number=		{1},
	pages=		{39 -- 79},
	issn=		{1432-1297}
}

\bib{Duistermaat_ActaMath_1972}{article}{
	author=		{Duistermaat, J. J.},
	author=		{H\"{o}rmander, L.},
	fjournal=	{Acta Mathematica},
	journal=	{Acta Math.},
	pages=		{183 -- 269},
	publisher=	{Institut Mittag-Leffler},
	title=		{Fourier integral operators. {I}{I}},
	volume=		{128},
	year=		{1972}
}

\bib{Guillemin_AdvMath_1985}{article}{
	title=		{A new proof of {W}eyl's formula on the asymptotic distribution of eigenvalues},
	journal=	{Adv. Math.},
	volume=		{55},
	number=		{2},
	pages=		{131 -- 160},
	year=		{1985},
	issn=		{0001-8708},
	author=		{Guillemin, V.}
}

\bib{Guillemin_InternationalP_2013}{book}{
	title=		{Semi-Classical Analysis},
	author=		{Guillemin, V.},
	author=		{Sternberg, S.},
	isbn=		{9781571462763},
	publisher=	{International Press of Boston, Inc.},
	address=	{US}, 
	year= 		{2013}
}


\bib{Gutzwiller_JMP_1971}{article}{
	author=		{Gutzwiller, M. C.},
	title=		{Periodic Orbits and Classical Quantization Conditions},
	journal=	{J. Math. Phys.},
	volume=		{12}, 
	number=		{3},
	pages=		{343 -- 358},
	year=		{1971},
}

\bib{Hoermander_ActaMath_1968}{article}{
	author=	{H\"{o}rmander, L.},
	title=	{The spectral function of an elliptic operator},
	journal={Acta Math.},
	year=	{1968},
	volume=	{121},
	number=	{1},
	pages=	{193 -- 218},
	issn=	{1871-2509}
}

\bib{Hoermander_ActaMath_1971}{article}{
	author=		{H\"{o}rmander, L.},
	fjournal=	{Acta Mathematica},
	journal=	{Acta Math.},
	pages=		{79 -- 183},
	publisher=	{Institut Mittag-Leffler},
	title=		{Fourier integral operators. {I}},
	volume=		{127},
	year=		{1971}
}

\bib{Hoermander_Springer_2007}{book}{
	author=			{H\"{o}rmander, L.},
	title=			{The Analysis of Linear Partial Differential Operators {I}{I}{I}: Pseudo-Differential Operators},
	series=			{Classics in Mathematics},
	volume=			{},
	edition=		{},
	publisher=		{Springer-Verlag Berlin Heidelberg},
	address=		{Germany},
	year=			{2007}
}

\bib{Hoermander_Springer_2009}{book}{
	author=			{H\"{o}rmander, L.},
	title=			{The Analysis of Linear Partial Differential Operators {I}{V}: {F}ourier Integral Operators},
	series=			{Classics in Mathematics},
	volume=			{},
	edition=		{},
	publisher=		{Springer-Verlag Berlin Heidelberg},
	address=		{Germany},
	year=			{2009}
}

\bib{Islam}{article}{
	title={On microlocalization and the construction of {F}eynman propagators for normally hyperbolic operators}, 
	author={Islam, O.},
	author={Strohmaier, A.},
	year={2020},
	eprint={arXiv:2012.09767[math.AP]},
	archivePrefix={arXiv},
	primaryClass={math.AP}
}

\bib{Ivrii_BullMathSci_2016}{article}{
	title=         {100 years of {W}eyl’s law},
	author=        {Ivrii, V.},
	journal=       {Bull. Math. Sci.},
	volume=        {6},
	pages=         {379 -- 452},
	year=          {2016},
	eprint=        {arXiv:1608.03963[math.SP]},
	archivePrefix= {arXiv},
	primaryClass=  {math.SP}
}

\bib{Jakobson_CMP_2007}{article}{
	author=			{Jakobson, D.},
	author= 		{Strohmaier, A.},
	title=			{High Energy Limits of {L}aplace-Type and {D}irac-Type EigenFunctions and Frame Flows},
	journal=		{Commun. Math. Phys.},
	year=			{2007},
	month=			{Mar},
	day=			{01},
	volume=			{270},
	number=			{3},
	pages=			{813 \ndash 833},
	issn=			{1432-0916},
	url=			{https://doi.org/10.1007/s00220-006-0176-0},
	eprint=			{arXiv:0607616v1 [math.SP]}
}

\bib{Khesin_AdvMath_2009}{article}{
	title=		{Pseudo-{R}iemannian geodesics and billiards},
	journal=	{Adv. Math.},
	volume=		{221},
	number=		{4},
	pages=		{1364 -- 1396},
	year=		{2009},
	issn=		{0001-8708},
	url=		{http://www.sciencedirect.com/science/article/pii/S0001870809000553},
	author=		{Khesin, B.},
	author=		{Tabachnikov, S.},
	eprint=		{arXiv:math/0608620[math.DG]}
}

\bib{Laptev_CPAM_1994}{article}{
	author= 	{Laptev, A.},
	author=		{Safarov, Y.},
	author=		{Vassiliev, D.},
	title= 		{On global representation of lagrangian distributions and solutions of hyperbolic equations},
	journal= 	{Commun. Pure Appl. Math.},
	volume= 	{47},
	number= 	{11},
	pages= 		{1411 \ndash 1456},
	url= 		{https://onlinelibrary.wiley.com/doi/abs/10.1002/cpa.3160471102},
	year= 		{1994}
}


\bib{Li_JGP_2016}{article}{
	title=		{The local counting function of operators of {D}irac and {L}aplace type},
	journal=	{J. Geom. Phys.},
	volume=	{104},
	pages=		{204 -- 228},
	year=		{2016},
	issn=		{0393-0440},
	author=	{Li, L.},
	author={Strohmaier, A.},
	eprint       = {arXiv:1509.00198[math.SP]},
	archivePrefix= {arXiv},
	primaryClass = {math.SP}
}

\bib{McCormick}{article}{
	author = {McCormick, A.},
	title = {A Trace Formula on Stationary {K}aluza-{K}lein Spacetimes},
	eprint={arXiv:2203.16729[math-ph]},
	year = {2022}
}

\bib{Meinrenken_ReptMathPhys_1992}{article}{
	title=		{Semiclassical principal symbols and {G}utzwiller's trace formula},
	journal=	{Rep. Math. Phys.},
	volume=		{31},
	number=		{3},
	pages=		{279 -- 295},
	year=		{1992},
	issn=		{0034-4877},
	author=		{Meinrenken, E.}
}

\bib{Meinrenken_JGP_1994}{article}{
	title=	{Trace formulas and the {C}onley-{Z}ehnder index},
	journal={J. Geom. Phys.},
	volume=	{13},
	number=	{1},
	pages=	{1 -- 15},
	year=	{1994},
	issn=	{0393-0440},
	author=	{Meinrenken, E.}
}

\bib{Muehlhoff_JMP_2011}{article}{
	author=        {M\"{u}hlhoff, R.},
	title=         {Cauchy Problem and {G}reen's Functions for First Order Differential Operators and Algebraic Quantization},
	eprint=        {arXiv:1001.4091[math-ph]},
	archivePrefix= {arXiv},
	primaryClass=  {math-ph},
	journal=       {J. Math. Phys.},
	volume=        {52},
	pages=         {022303},
	year={2011}
}

\bib{Muratore-Ginanneschi_PR_2003}{article}{
	title=	{Path integration over closed loops and {G}utzwiller's trace formula},
	journal=	{Phys. Rep.},
	volume=	{383},
	number=	{5},
	pages=	{299 -- 397},
	year=		{2003},
	issn=		{0370-1573},
	author=	{Muratore-Ginanneschi, P.},
	archivePrefix= {arXiv},
	eprint= 	{arXiv:0210047[nlin-cd]},
	primaryClass= {nlin-cd}
}

\bib{Penrose_1972}{article}{
	author=	{Penrose, R.}, 
	title=	{On the nature of quantum geometry},
	book={
		title={Magic Without Magic},
		editor=	{Klauder, J. R.},
		address={San Francisco},
		publisher=	{W. H. Freeman},
		date={1972},
	}
}

\bib{Safarov_JFA_2001}{article}{
	title=         {Fourier Tauberian Theorems and Applications},
	journal=       {J. Funct. Anal.},
	volume=        {185},
	number=        {1},
	pages=         {111 -- 128},
	year=          {2001},
	issn=          {0022-1236},
	author=        {Safarov, Y.}, 
	archivePrefix= {arXiv},
	eprint= 	   {arXiv:0003014[math.SP]},
	primaryClass=  {math.SP}
}

\bib{Safarov_AMS_1997}{book}{
	title=    {The Asymptotic Distribution of Eigenvalues of Partial Differential Operators},
	author=   {Safarov, Y.},
	author=	  {Vassiliev, D.},
	volume=   {155},
	series=   {Translations of Mathematical Monographs},
	year=     {1997},
	address=  {USA},
	publisher={American Mathematical Society}
}



\bib{Sanchez_AMS_1999}{article}{
	title={Timelike periodic trajectories in spatially compact Lorentz manifolds},
	author={S\'{a}nchez, M.},
	journal={Proc. Am. Math. Soc.},
	volume={127},
	number={10},
	pages={3057--3066},
	year={1999}
}

\bib{Sanchez_AMS_2011}{article}{
	author=			{S\'{a}nchez, M.},
	title=			{Recent Progress on the Notion of Global Hyperbolicity},
	book={
		editor=		{Plaue, M.},
		editor=		{Rendall, A.},
		editor=		{Scherfner, M.},
		title=		{Advances in {L}orentzian Geometry: Proceedings of the {L}orentzian Geometry Conference in {B}erlin},
		year=		{2011},
		series=		{AMS/IP Studies in Advanced Mathematics},
		volume=		{49},
		publisher=	{American Mathematical Society and International Press},
		address=	{USA},
	}
	pages=			{105 -- 124},
	doi=		    {},
	archivePrefix= 	{arXiv},
	eprint       = 	{arXiv:0712.1933[gr-qc]},
	primaryClass = 	{gr-qc}
}

\bib{Sandoval_CPDE_1999}{article}{
	author=		{Sandoval, M. R.},
	title=		{Wavetrace asymptotics for operators of {D}irac type},
	journal=	{Commun. Partial Differ. Equ.},
	volume=		{24},
	number=		{9-10},
	pages=		{1903-1944},
	year=		{1999},
	publisher=	{Taylor \& Francis}
}

\bib{Strohmaier_AdvMath_2021}{article}{ 
	author=			{Strohmaier, A.},
	author=			{Zelditch, S.},
	title=			{A {G}utzwiller trace formula for stationary space-times},
	journal=		{Adv. Math.},
	volume=			{376},
	number=			{},
	pages=			{107434},
	year=			{2021},
	archivePrefix=	{arXiv},
	eprint=			{arXiv:1808.08425[math.AP]},
	primaryClass=	{math.AP}
}

\bib{Strohmaier_IndagMath_2021}{article}{
	title=			{Semi-classical mass asymptotics on stationary spacetimes},
	journal=		{Indag. Math.},
	volume=			{32},
	number=			{1},
	pages=			{323--363},
	year=			{2021},
	issn=			{0019-3577},
	author=			{Strohmaier, A.},
	author=			{Zelditch, S.},
	archivePrefix=	{arXiv},
	eprint=			{arXiv:2002.01055[math-ph]},
	primaryClass=	{math-ph}
}

\bib{Strohmaier_RMP_2021}{article}{
	author=		{Strohmaier, A.},
	author=		{Zelditch, S.},
	title=		{Spectral asymptotics on stationary space-times},
	journal=	{Rev. Math. Phys.},
	volume=		{33},
	number=		{01},
	pages=		{2060007},
	year=		{2021}
}

\bib{Toth_GFA_2013}{article}{
	title=			{Quantum Ergodic Restriction Theorems: Manifolds Without Boundary},
	journal=		{Geom. Funct. Anal.},
	volume=			{23},
	pages=			{715 -- 775},
	year=			{2013},
	issn=			{0926-2245},
	author=			{Toth, J. A.},
	author=			{Zelditch, S.},
	archivePrefix=	{arXiv},
	eprint       =	{arXiv:1104.4531[math.SP]},
	primaryClass =	{math.SP}
}

\bib{Treves_Plenum_1980}{book}{
	title=		{Introduction to Pseudodifferential and Fourier Integral Operators: Fourier Integral Operators},
	author=		{Treves, J.-F.},
	volume=		{2},
	series=		{University Series in Mathematics},
	year=		{1980; Second Printing 1982},
	address=	{New York}, 
	publisher=	{Plenum Press}
}

\bib{Uribe_Cuernavaca_1998}{article}{
	author=	{Uribe, A.}, 
	title=	{Trace Formulae},
	conference={
		title={First Summer School in Analysis and Mathematical Physics: Quantization, the 		{S}egal-{B}argmann Transform and Semiclassical Analysis},
		address={Cuernavaca Morelos, Mexico},
		data={1998},
		},
	book={
		editor=	{P\'{e}rez-Esteva, S.},
		editor={Villegas-Blas, C.}, 
		series={Contemporary Mathematics},
		volume={260},
		publisher=	{American Mathematical Society},
		date={2000},
		},
	Pages=	{61 -- 90}
}

\bib{Verdiere_AIF_2007}{article}{
	author= 	{Colin de Verdi\`{e}re, Y.},
	title= 		{Spectrum of the {L}aplace operator and periodic geodesics: thirty years after},
	journal= 	{Ann. Inst. Fourier},
	publisher= 	{Association des Annales de l'institut Fourier},
	volume= 	{57},
	number= 	{7},
	year= 		{2007},
	pages= 		{2429--2463}
}

\bib{Wunsch_Zuerich_2008}{book}{
	title=			{Microlocal analysis and evolution equations: Lecture Notes from the 2008 {C}{M}{I}/{E}{T}{H} Summer School},
	author=		{Wunsch, J.},
	year=			{2013}, 
	series=		{Clay Mathematics Proceedings},
	volume=		{17},
	publisher=		{American Mathematical Society},
	pages=			{1--72},
	editor=		{Ellwood, D.},
	editor= {Rodnianski, I.},
	editor={Staffilani, G.},
	editor={Wunsch, J.},
	booktitle=		{Evolution Equations}, 
	url=			{https://bookstore.ams.org/cmip-17},
	eprint=		{arXiv:0812.3181[math.AP]}
}

\end{biblist}
\end{bibdiv}
\end{document}